\documentclass[12p]{article}
\usepackage{bbm}
\usepackage[cp1250]{inputenc}
\usepackage[T1]{fontenc}
\usepackage[reqno]{amsmath}
\usepackage{amssymb}
\usepackage{amsthm}
\usepackage{hyperref}
\usepackage[margin=1in]{geometry}
\numberwithin{equation}{section}

\newtheorem{theorem}{Theorem}[section]

\newtheorem{lemma}[theorem]{Lemma}

\theoremstyle{definition}

\newtheorem{remark}[theorem]{Remark}

\usepackage{blindtext}
\usepackage[affil-it]{authblk}

\newcommand{\supp}{\operatorname{supp}}
\newcommand {\<}{\left\langle}  
\renewcommand {\>}{\right\rangle}  

\title{The Strassen Invariance Principle for Certain
Non-stationary Markov-Feller Chains}
\newcommand\CoAuthorMark{\footnotemark[\arabic{footnote}]}
\author[  \hspace{-1ex}]{Dawid Czapla\protect\CoAuthorMark
  \thanks{Corresponding author;\; e-mail address: \texttt{dawid.czapla@us.edu.pl}}, Katarzyna Horbacz and Hanna Wojew\'odka-\'Sci\k{a}\.zko}
\affil[  \hspace{-1ex}]{\small\textit{Institute of Mathematics, University of Silesia in Katowice, Bankowa 14, 40-007 Katowice, Poland}} 
\date{}

\begin{document}
\maketitle
\vspace*{-1.2cm}
\begin{abstract}
We propose certain conditions implying the functional law of the iterated logarithm (the Strassen invariance principle) for some general class of non-stationary Markov-Feller chains. This class may be briefly specified by the following two properties: firstly, the transition operator of the chain under consideration enjoys a non-linear Lyapunov-type condition, and secondly, there exists an~appropriate Markovian coupling whose transition probability function can be decomposed into two parts, one of which is contractive and dominant in some sense. Our criterion may serve as a useful tool in verifying the functional law of the iterated logarithm for certain random dynamical systems, developed e.g. in biology and population dynamics. In the final part of the paper we present an example application of our main theorem to a mathematical model describing stochastic dynamics of gene expression.
\end{abstract}
{\small
}

\section*{Introduction}

The law of the iterated logarithm (LIL) can be viewed as a~refinement of the strong law of large numbers (SLLN). 
It improves the convergence rate in the SLLN from $\mathcal{O}(t)$ to  $\mathcal{O}(\ln(\ln(t)))$. More specifically, it provides the precise values of the lower and upper limit of almost all sequences formed by the properly scaled partial sums (or integrals) of the sample paths of the stochastic process under study. 
Moreover, the LIL gives an interesting illustration of the difference between almost sure and distributional statements, such as the central limit theorem (CLT).

The functional version of the LIL, now usually called the Strassen invariance principle, was first proven for sums of independent and identically distributed random variables by V.~Strassen  (cf. \cite{strassen}). 
Later, it was extended to square integrable martingales (see e.g.~\cite{hh,hs}) and also to certain particular classes of Markov chains, including stationary processes (cf.~\cite{wu, zw}), as well as non-stationary ones. The latter include, for instance, positive Harris recurrent Markov chains with drift towards petite sets \hbox{(cf.~\cite[Theorem 17.5.3]{mt})} or Markov-Feller chains enjoying the exponential mixing property in the Wasserstein metric (see \cite{bms}). At this point, it is worth stressing that the well-known techniques developed by S.P. Meyn and R.L. Tweedie \cite{mt} are usually only applicable if the state space of the examined Markov chain is locally compact, which is not the case neither here, nor in \cite{bms}.

The main result of this paper is a version of the Strassen invariance principle for a~quite general class of non-stationary Markov-Feller chains evolving on Polish spaces. However, on the contrary to \cite{bms}, 
we use a form of exponential mixing (in the Fortet--Mourier metric; cf. \cite{dawid,clt_chw}), in which a distance between measures $P^n \mu_1$ and $P^n \mu_2$, where $P$ is a transition operator of a Markov chain, does not depend on a~distance between initial distributions $\mu_1$ and $\mu_2$ (similarly as in \cite{hairer}). More precisely, the overall strategy of the proof of our main result (Theorem~\ref{thm:LIL}) is based on the condition of the form
\begin{equation}\label{our_cond}
d_{FM}\left(P^n \mu_1, P^n \mu_2\right) \leq C q^n \left(1+ \int V \,d(\mu_1+\mu_2)\right)\;\;\;\mbox{for}\;\;\; n\in\mathbb{N}\;\;\;\mbox{and some}\;\;\;q<1,
\end{equation}
rather than 
\begin{equation}\label{old_cond}
d_{FM}\left(P^n \mu_1, P^n \mu_2\right) \leq C q^n \,d_{FM}(\mu_1, \mu_2)\;\;\;\mbox{for}\;\;\; n\in\mathbb{N}\;\;\;\mbox{and some}\;\;\;q<1,
\end{equation}
where $\varrho$ and $d_{FM}$ stand for a metric on a state space and the above-mentioned Fortet--Mourier distance, respectively, while $V$ is a Lyapunov function.

It is also worth stressing that we do not assume condition \eqref{our_cond} directly. Instead, we propose a set of conditions, relatively easy to verify, which yield \eqref{our_cond} and the desired assertion.  The motivation to establish such a result derives from our research on certain random dynamical systems, applied mainly in molecular biology (see e.g. the models for gene expression investigated in \cite{hhs, dawid, mtky} or the model for cell cycle discussed in \cite{lm, hw}), to which we have not been able to apply  \hbox{\cite[Theorem 1]{bms}} directly. This is primarily caused by the fact that, upon certain general conditions imposed on the model (which appear to be reasonable in most applications), \eqref{old_cond} seems to be difficult or even impossible to achieve, whilst the same conditions naturally imply \eqref{our_cond}, as shown e.g. in \cite[Theorem 4.1]{dawid}.

The class of Markov-Feller chains for which we state our main result (Theorem~\ref{thm:LIL}), that is, the Strassen invariance principle for the LIL, can be 
characterized briefly by the following two properties. Firstly, the transition operator of the chain under consideration enjoys a non-linear Lyapunov-type condition. Secondly, there exists an appropriate Markovian coupling, whose transition function can be decomposed into two parts, one of which is contractive and dominant in some sense. The construction of such a~coupling is described in details e.g. in \cite{hairer, hw, ks, sleczka}. Some proof techniques, employed in this paper, are adapted from the articles \cite{bms} and \cite{hhsw2}, which both pertain to the martingale results by C.C.~Heyde and D.J. Scott \cite{hs}. One of the simplest classes of Markov chains achieving the desired properties are those arised from random iterated function systems with an arbitrary number of transformations, which are assumed to be contractive on average, such as those considered in \cite{ks,sleczka, werner, horbacz_szarek}. 

Our main result is formulated in the same spirit as \cite[Theorem 2.1]{ks} and \hbox{\cite[Theorem 2.1]{clt_chw}}, whose applicability was illustrated by proving the exponential ergodicity (in the Fortet--Mourier distance) and the CLT, respectively, for some, important from the application point of view, random dynamical system (cf. \cite{dawid, asia}). Here we use our generel result to establish the functional LIL for such a model (cf. Theorem~\ref{thm_LIL}). 

The aforementioned dynamical system has interesting biological interpretations. First of all, it can be viewed as the Markov chain given by the post-jump locations of some piecewise-deterministic Markov process, discussed in Section \ref{sec:ex}, which
occurs in a~simple model of gene expression (cf.~\cite{dawid, mtky}). Incidentally, this process can be also identified with the solution of a Poisson driven stochastic differential equation (in the context cosidered in \cite{horbacz_poiss, kazak, asia}), mainly developed by A. Lasota and J. Traple \cite{las_traple}. 
On the other hand, a special case of the above-mentioned abstract random dynamical system, defined as an iterated function system with an additive perturbation (see \cite{hhsw2}), provides a mathematical framework for modelling the concentration of the compunds involved in the gene autoregulation at times of transcriptional bursts (for details, see \cite{hhs}). The latter example indicates the importance of considering a non-locally compact space as the state space in the abstract framework. 
Furthermore, it is also worth mentioning that in the case where no disturbance is present, we obtain an ordinary random iterated function system (with an arbitrary set of transformations), which applies e.g. in a model of cell cycle (cf. e.g. \cite{lm,hw}). 

Finally, let us point out, that, beside the examples captured by the model discussed in Section \ref{sec:ex}, there exist other dynamical systems, such as e.g. the one considered in [16],  which fits the abstract framework of Theorem \ref{thm:LIL}, but cannot be obtained as the special case of the aforementioned model.

The article is organised as follows. In Section \ref{sec:1}, we gather notation and definitions used throughout the paper. We mainly relate to the general theory of Markov chains, discussed more widely
e.g. in \hbox{\cite{mt, revuz}}, and, in particular, we introduce the concept of a Markovian coupling. In Section \ref{sec:SGP}, we quote some auxiliary results, established in \cite{ks,clt_chw}, while in Section~\ref{sec:LIL}, we formulate and prove our main result. 
At the beginning of this section, we also present a few general observations concerning martingales, whose proofs are carried out in \hyperref[append]{Appendix}. 
Finally, in Section \ref{sec:ex}, we apply our main result to the above-mentioned particular dynamical system (considered e.g. in \cite{dawid}), related to~a model of gene expression.

\section{Prelimenaries}\label{sec:1}

In the beginning, we shall introduce some notation and recall certain general definitions, as well as basic facts that will be used in our further analysis.

Let us write $\mathbb{R}_+=[0,\infty)$ and $\mathbb{N}_0=\mathbb{N}\cup\{0\}$ with $\mathbb{N}$ standing for the set of all positive integers. For any point $x$ and any set $A$, the symbols $\delta_x$ and $\mathbbm{1}_A$ will denote the Dirac measure at $x$ and the indicator function of $A$, respectively.

We consider a complete separable metric space $(X,\varrho)$, endowed with the \hbox{$\sigma$-field} $\mathcal{B}_X$ of its Borel subsets. By $B_b(X)$ we will denote the space of all bounded Borel measurable functions $f:X\to\mathbb{R}$, equipped with the supremum norm $\|f\|_{\infty}=\sup_{x\in X}|f(x)|$, while $C_b(X)$ and $Lip_b(X)$ will stand for the subspaces of $B_b(X)$ consisting of all continuous and all Lipschitz continuous functions, repectively. In the present paper we shall also refer to  the space $\bar{B}_b(X)$ consisting of functions $f:X\to\mathbb{R}$ which are Borel measurable and bounded below.

In what follows, we will write $\mathcal{M}_{fin}(X)$ and $\mathcal{M}_{1}(X)$ for the sets of all finite and all probability Borel measures on $X$, respectively. We shall also introduce  
\begin{equation*}
\mathcal{M}_{1,r}^V(X)=\left\{\mu\in \mathcal{M}_1(X):\; \int_XV^r(x)\,\mu(dx)<\infty\right\}\;\;\;\mbox{for any}\;\;\;r>0
\end{equation*}
and any given Lyapunov function $V:X\to[0,\infty)$, that is, a continuous function which is  bounded on bounded sets, and, in the case of unbounded $X$, satisfies \hbox{$\lim_{\varrho(x,\bar{x})\to\infty}V(x)=\infty$} for some $\bar{x}\in X$. 
For brevity, for any $f\in \bar{B}_b(X)$ and any signed Borel measure $\mu$ on $X$, we will write $\langle f,\mu\rangle$ for $\int_Xf(x)\,\mu(dx)$. As usual, $\mbox{supp}\,\mu$ will denote the support of $\mu\in\mathcal{M}_{fin}(X)$. 

To evaluate the distance between probability measures, we will use the so-called Fortet--Mourier distance (see e.g. \cite{l_frac}), defined as follows:
\begin{equation*}
d_{FM}(\mu_1,\mu_2)=\sup\left\{\left|\left\langle f,\mu_1-\mu_2\right\rangle\right|:\;f\in Lip_{FM}(X)\right\}\quad\mbox{for}\quad \mu_1,\mu_2\in \mathcal{M}_1(X),
\end{equation*}
where  
\begin{equation*}
Lip_{FM}(X)=\{f\in C_b(X):\;\|f\|_{BL}\leq 1\}, \;\;\;\|f\|_{BL}=\max\{|f|_{Lip},\,\|f\|_{\infty}\}
\end{equation*}
and $|f|_{Lip}$ stands for the minimal Lipschitz constant of $f$.  It is well-known that, whenever $(X,\varrho)$ is a~complete separable metric space, which is the case here, the convergence in $d_{FM}$ is equivalent to the weak convergence of probability measures. Moreover, upon this assumption, the space $(\mathcal{M}_1(X),d_{FM})$ is complete (see \cite{dudley} for the proofs of both these facts).

A mapping $\Pi:X\times \mathcal{B}_X\to [0,1]$ is called a (sub)stochastic kernel if \hbox{$\Pi(\cdot,A):X\to[0,1]$} is a Borel measurable function for any fixed $A\in \mathcal{B}_X$, and \hbox{$\Pi(x,\cdot):\mathcal{B}_X\to[0,1]$} is a~(sub)probability Borel measure for any fixed $x\in X$. Given a (sub)stochastic kernel $\Pi$, we can also define the $n$-th step kernels $\Pi^n$, $n\in\mathbb{N}_0$, by setting, for every $x\in X$ and any $A\in\mathcal{B}_X$,
\begin{equation*}
\Pi^0(x,A)=\delta_x(A)\;\;\;\text{and}\;\;\; 
\Pi^n(x,A)=\int_X\Pi(y,A)\,\Pi^{n-1}(x,dy)
\;\;\; \mbox{for}\;\;\;
n\geq 1.
\end{equation*}

Every stochastic kernel $\Pi$ naturally induces a Markov operator \hbox{$P:\mathcal{M}_{fin}(X)\to \mathcal{M}_{fin}(X)$} and its dual operator \hbox{$U:{B}_b(X)\to {B}_b(X)$}, which are given by the formulas:
\begin{align}
\label{def:markov_op}
&P\mu(A)=\int_X\Pi(x,A)\,\mu(dx)\;\;\mbox {for} \;\;\mu\in \mathcal{M}_{fin}(X),\; A\in \mathcal{B}_X,\\
\label{def:dual_op}
&Uf(x)=\int_Xf(y)\,\Pi(x,dy) \;\;\mbox{for}\;\; f\in {B}_b(X),\;x\in X.
\end{align}
By the duality of operators $P$ and $U$ we mean the following relationship:
\begin{equation}\label{eq:duality}
\langle f,P\mu \rangle=\langle Uf,\mu\rangle \quad\mbox{for}\quad f\in{B}_b(X),\; \mu\in \mathcal{M}_{fin}(X).
\end{equation}
Let us note that $U$, given by \eqref{def:dual_op}, can be extended in the usual way to the space $\bar{B}_b(X)$ so that \eqref{eq:duality} holds for all functions $f$ from this space.

Let $P$ be an arbitrary Markov operator defined as in \eqref{def:markov_op}. If $P\mu_*=\mu_*$ for some $\mu_*\in \mathcal{M}_{fin}(X)$, then $\mu_*$ is called an~invariant measure of $P$. The operator $P$ is said to be exponentially ergodic in~$d_{FM}$ (on the set $\mathcal{M}_{1,1}^V(X)$) whenever it has a~unique invariant measure $\mu_*\in \mathcal{M}_{1,1}^V(X)$ and there exists $q\in (0,1)$ such~that
\begin{equation*}
d_{FM}(P^n\mu,\mu_*)\leq q^n c(\mu)\;\;\;\mbox{for any}\;\;\;\mu\in\mathcal{M}_{1,1}^V(X),\;n\in\mathbb{N},
\end{equation*}
where $c(\mu)$ is a constant depending on $\mu$.

Suppose now that $(\phi_n)_{n\in\mathbb{N}_0}$ is an $X$-valued time-homogeneous Markov chain defined on a~probability space $(\Omega, \mathcal{A},\mathbb{P})$. Then the formula
\begin{equation} \label{trans_n} 
\Pi(x,A):=\mathbb{P}(\phi_{n+1}\in A|\phi_n=x)\;\;\;\mbox{for}\;\;\; x\in X,\;A\in \mathcal{B}_X,\;n\in\mathbb{N}_0
\end{equation}
defines a stochastic kernel, which determines the so-called one-step transition law of the chain $(\phi_n)_{n\in\mathbb{N}_0}$. The evolution of the distributions \hbox{$\mu_n(\cdot):=\mathbb{P}(\phi_n\in \cdot)$} can be then described by the Markov \hbox{operator} $P$ induced by $\Pi$ (called a transition operator in this context), i.e. $\mu_{n+1} =P\mu_n$ for any $n\in\mathbb{N}_0$.

On the other hand, for any given stochastic kernel $\Pi$ and any probability measure $\mu\in\mathcal{M}_1(X)$, we can always define a time-homogeneous Markov chain $(\phi_n)_{n\in \mathbb{N}_0}$ with transition law $\Pi$ and initial measure $\mu$ as a coordinate process on the space $\Omega:=X^{\mathbb{N}_0}$ endowed with the product topology. More specifically $(\phi_n)_{n\in\mathbb{N}_0}$ is then a sequence of projections from $\Omega$ to $X$, given by $\phi_n(\omega)=x_n$ for \hbox{$\omega=(x_0,x_1,\ldots)\in \Omega$}. In this case, according to \hbox{\cite[Theorem 3.4.1]{mt}}, there exists a~probability measure $\mathbb{P}_{\mu}\in\mathcal{M}_1(\Omega)$ 
such that, for any $n\in\mathbb{N}_0$ and any $A_0,\ldots,A_n\in\mathcal{B}_X$, we have
\begin{equation}\label{eq:P_mu}
\mathbb{P}_{\mu}(A_0\times\ldots\times A_n\times X \times\ldots)=\int_{A_0}\int_{A_1}\ldots \int_{A_{n-1}}P(x_{n-1},A_n)P(x_{n-2},dx_{n-1})P(x_0,dx_1)\mu(dx_0).\;\;
\end{equation}
It can be shown that $(\phi_n)_{n\in \mathbb{N}_0}$ is then a~time-homogeneous Markov chain on the probability space $(\Omega,\mathcal{B}_{\Omega},\mathbb{P}_{\mu})$ with transition law $\Pi$ and initial distribution $\mu$. Clearly, 
$\mathbb{P}_{\mu}(B)$ is then the probability of the event $\left\{(\phi_n)_{n\in\mathbb{N}_0}\in B\right\}$ for $B\in\mathcal{B}_{\Omega}$. 
The Markov chain defined according to the above scheme will be further called a canonical Markov chain. 
The expectation operator corresponding to $\mathbb{P}_{\mu}$ will be denoted, as usual, by $\mathbb{E}_{\mu}$. Moreover, by convention, for any $x\in X$, we will write $\mathbb{P}_{x}$ and $\mathbb{E}_x$ rather than $\mathbb{P}_{\delta_x}$ and  $\mathbb{E}_{\delta_x}$, respectively. Obviously, one can easily check that 
\begin{equation*}
\mathbb{P}_x(B)=\mathbb{P}_{\mu}(B \, | \, \phi_0=x)\quad\text{for any}\quad B\in\mathcal{B}_{\Omega}\quad\text{and}\quad\mu\in\mathcal{M}_1(X).
\end{equation*}

A~time-homogeneous Markov chain $(\phi^{(1)}_n,\phi^{(2)}_n)_{n\in\mathbb{N}_0}$ evolving on $X^2$ (endowed with the product topology) is said to be a Markovian coupling of some stochastic kernel $\Pi$ whenever its transition law $C:X^2\times  \mathcal{B}_{X^2}\to\left[0,1\right]$ satisfies
\begin{equation*}
C(x,y,A\times X)= \Pi(x,A)\;\;\;\text{and}\;\;\;
C(x,y,X\times A)= \Pi(y,A)\;\;\;\mbox{for any}\;\;\; x,y\in X,\;A\in \mathcal{B}_X.
\end{equation*}
Conventionally, the kernel $C$ itself is often called a coupling of $\Pi$, too.

In practise, given a measure $\alpha\in\mathcal{M}_1(X^2)$, it is convenient to consider the canonical form of the coupling $(\phi_n^{(1)}, \phi_n^{(2)})_{n\in\mathbb{N}_0}$, defined on the coordinate space $((X^2)^{\mathbb{N}_0},\mathcal{B}_{(X^2)^{\mathbb{N}_0}})$ endowed with an appropriate probability measure $\mathbb{C}_{\alpha}$, which makes $\alpha$ the initial distribution of this chain and obeys the rule corresponding to \eqref{eq:P_mu} with $C$ and $\mathbb{C}$ in the roles of $P$ and $\mathbb{P}$, respectively. 
In accordance with the convention adopted above, we will use the symbol $\mathbb{C}_{x,y}$ instead of $\mathbb{C}_{\alpha}$ in the case where $\alpha=\delta_{(x,y)}$ for some $(x,y)\in X^2$. The expected values corresponding to $\mathbb{C}_{\alpha}$ and $\mathbb{C}_{x,y}$ will be denoted by~$\mathbb{E}_{\alpha}$ and $\mathbb{E}_{x,y}$, respectively.

Let us also indicate that, for any stochastic kernel $\Pi:X\times\mathcal{B}
_X\to[0,1]$ and any substochastic kernel \hbox{$Q:X^2\times\mathcal{B}_{X^2}\to [0,1]$} satisfying 
\begin{equation}\label{def:substoch}
Q(x,y,B\times X)\leq \Pi(x,B)\;\;\;\text{and}\;\;\;
Q(x,y,X\times B)\leq\Pi(y,B)
\;\;\;\mbox{for}\;\;\; x,y\in X,\;B\in\mathcal{B}_X,
\end{equation}
there exists a substochastic kernel $R:X^2\times \mathcal{B}_{X^2}\to[0,1]$ such that $C=Q+R$ is a~Markovian coupling of $\Pi$ (see e.g. \cite{dawid,ks,hw} for the explicit formula of $R$).

\section{Conditions Sufficient for Exponential Ergodicity}\label{sec:SGP}

Consider a stochastic kernel $\Pi:X\times\mathcal{B}_X\to[0,1]$, and let $P$, $U$ be the operators given by \eqref{def:markov_op}, \eqref{def:dual_op}, respectively. Below, and throughout the rest of this paper, we will impose the following assumptions:
\begin{itemize}
\item[(B0)] \phantomsection \label{cnd:B0} The Markov operator $P$ has the Feller property.
\item[(B1)] \phantomsection \label{cnd:B1} There exist a Lyapunov function $V:X\to[0,\infty)$ and constants $a\in (0,1)$, $b\in(0,\infty)$ such that
\begin{equation*}\langle V, P\mu\rangle\leq a \langle V,\mu\rangle +b \quad\mbox{for every}\quad \mu\in \mathcal{M}_{1,1}^V(X).\end{equation*}  
\end{itemize}
Furthermore, we will also require the existance of a~substochastic kernel \hbox{$Q:X^2\times\mathcal{B}_{X^2}\to[0,1]$} which satisfies  (\ref{def:substoch}) and, for some $F\subset X^2$, enjoys the following conditions:
\begin{itemize}
\item[(B2)] \phantomsection \label{cnd:B5} 
There exists 
a Markovian coupling $(\phi^{(1)}_n,\phi^{(2)}_n)_{n\in \mathbb{N}_0}$ of $\Pi$ with transition law $C\geq Q$ such that, for some $\Gamma>0$,     
we can choose $\gamma\in(0,1)$ and $c_{\gamma}>0$ for which
\begin{equation*}\mathbb{E}_{x,y}(\gamma^{-\rho})\leq c_{\gamma},\quad\mbox{whenever}\quad V(x)+V(y)<4b(1-a)^{-1},\end{equation*}
where
\begin{equation}\label{def:tau}
\rho
:=\inf\left\{n\in\mathbb{N}_0:\;(\phi^{(1)}_n,\phi^{(2)}_n)\in F\;\mbox{and}\;V\left(\phi^{(1)}_n\right)+V\left(\phi^{(2)}_n\right)<\Gamma\right\}.
\end{equation}
\item[(B3)] \phantomsection \label{cnd:B2} There exists  $\delta\in(0,1)$ such that 
\begin{equation*}
\mbox{supp}\,Q(x,y,\cdot)\subset F\;\;\;\mbox{and}\;\;\;
\int_{X^2}\varrho(u,v)\,Q(x,y,du\times dv)\leq \delta\, \varrho(x,y)\quad\mbox{for any}\quad(x,y)\in F.
\end{equation*}
\item[(B4)] \phantomsection \label{cnd:B3} Letting $U(r)=\{(u,v)\in F:\varrho(u,v)\leq r\}$ for any $r>0$, we have 
\begin{equation*}
\inf_{(x,y)\in F}Q\big(x,y,U\left(\delta\,\varrho(x,y)\right)\big)>0.
\end{equation*}
\item[(B5)] \phantomsection \label{cnd:B4}  There exist constants $\beta\in (0,1]$ and $c_{\beta}>0$ such that
\begin{equation*}
Q\left(x,y,X^2\right)\geq 1-c_{\beta}\,\varrho^{\beta}(x,y)\quad\mbox{for every}\quad (x,y)\in F.
\end{equation*}
\end{itemize}

Below, we quote two results that we extensively use in the present paper. They are proven in \cite{ks} and \cite{clt_chw}, respectively.

\begin{theorem}[\mbox{\cite[Theorem 2.1]{ks}}]\label{thm:spectral_gap}
Suppose that conditions \hyperref[cnd:B0]{(B0)}-\hyperref[cnd:B4]{(B5)} 
hold for a Markov operator $P$, some substochastic kernel $Q:X^2\times \mathcal{B}_{X^2}\to [0,1]$, satisfying \eqref{def:substoch}, and some $F\subset X^2$. 
Then, $P$ possesses a unique invariant measure $\mu_*\in \mathcal{M}_{1}(X)$ such that $\mu_*\in\mathcal{M}_{1,1}^V(X)$, where $V$ is the Lyapunov function  determined by~\hyperref[cnd:B1]{(B1)}. Moreover, there exist constants $q\in(0,1)$ and $c<\infty$ such that
\begin{equation}
\label{e:exp_erg}
d_{FM}\left(P^n\mu,\mu_*\right)\leq cq^n \left(1+\langle V,\mu\rangle+\langle V,\mu_*\rangle\right)\;\; \mbox{for any}\;\;\mu\in \mathcal{M}_{1,1}^V(X),\;n\in\mathbb{N}_0.
\end{equation}
\end{theorem}

\begin{lemma}[\mbox{\cite[Lemma 2.3]{clt_chw}}]\label{cor:g_useful}
Under the assumptions of Theorem \ref{thm:spectral_gap},  
there exist $q\in(0,1)$ and $c<\infty$ such that
\begin{equation}
\label{eq:lemma_thesis}
\mathbb{E}_{x,y}\left|g\left(\phi_n^{(1)}\right)-g\left(\phi_n^{(2)}\right)\right|\leq c\|g\|_{BL}\,q^n(1+V(x)+V(y))
\end{equation}
for all $(x,y)\in X^2$, $g\in Lip_b(X)$ and $n\in\mathbb{N}_0$, where the coupling $(\phi_n^{(1)}, \phi_n^{(2)})_{n\in\mathbb{N}_0}$ fulfills \hyperref[cnd:B5]{(B2)}.
\end{lemma}

The key idea underlying both Theorem \ref{thm:spectral_gap} and Lemma \ref{cor:g_useful}, in which conditions \hyperref[cnd:B0]{(B0)}-\hyperref[cnd:B4]{(B5)} are assumed, pertains to the so-called asymptotic coupling technique, introduced by M. Hairer in \cite{hairer}. Roughly speaking, conditions \hyperref[cnd:B1]{(B1)}-\hyperref[cnd:B4]{(B5)}  provide the existence of a Markovian coupling $(\phi_n^{(1)},\phi_n^{(2)})_{n\in\mathbb{N}_0}$ of~$\Pi$, whose transition function, say $C$, can be decomposed into two substochastic kernels, one of which, denoted by $Q$, enjoys the contractivity property, expressed by \hyperref[cnd:B2]{(B3)}, and plays a dominant role in the evolution of the coupling. The most technical condition \hyperref[cnd:B5]{(B2)} ensures that the dynamics under consideration quickly enters the set $F$, that is, the domain of contractivity of $Q$. By the dominance of $Q$ we mean the existence of an a.s. finite random time, say $\tau$, after which any further step of the coupled chain is drawn only according to $Q$.  Establishing this property involves the use of all hypotheses \hyperref[cnd:B1]{(B1)}-\hyperref[cnd:B4]{(B5)}. The dominant, contractive part $Q$ makes the copies of the Markov chain (governed by $P$) couple asymptotically at an exponential rate.

To better illustrate the main idea behind conditions \hyperref[cnd:B1]{(B1)}-\hyperref[cnd:B4]{(B5)}, let us sketch very briefly the proof of Lemma \ref{cor:g_useful}. The crucial point here is to consider an augmented coupling $(\hat{\phi}_n)_{n\in\mathbb{N}_0}$ of the form \hbox{$\hat{\phi}_n:=(\phi_n^{(1)}, \phi_n^{(2)}, \theta_n)$} with values in $X^2\times\{0,1\}$,  constructed in such a way that
\begin{align*}
&\mathbb{C}_{x,y}\left(\hat{\phi}_n\in A\times \{1\}\right)=Q^n(x,y, A),\;\;\;
\mathbb{C}_{x,y}\left(\hat{\phi}_n\in A\times \{0\}\right)=R^n(x,y, A),\\
&\mathbb{C}_{x,y}\left(\left(\phi_n^{(1)},\phi_n^{(2)}\right)\in A\right)=C^n(x,y,A)\;\;\;\text{for any}\;\;\;(x,y)\in X^2\;\;\;\text{and any}\;\;\;A\in\mathcal{B}_{X^2}.
\end{align*}
Then, the aforementioned random variable $\tau$ can be defined as an absorption time of the form
\begin{equation*}
\tau:=\inf\left\{n\in\mathbb{N}:\,\theta_m=1\;\;\;\text{for any}\;\;\; m\geq n\right\}.
\end{equation*}
Further, we get
\begin{align}
\label{s1}
\begin{split}
\mathbb{E}_{x,y}\left|f\left(\phi_n^{(1)}\right)-f\left(\phi_n^{(2)}\right)\right|
&\leq \|g\|_{BL}\int_{X^2} \varrho(u,v)\,\mathbb{C}_{x,y}\left(\rho^{(N)} \leq M,\,\tau\leq N,\,\left(\phi_n^{(1)},\phi_n^{(2)}\right)\in du\times dv\right)\\
&\quad+2\|g\|_{BL}\mathbb{C}_{x,y}\left(\rho^{(N)}>M\right)+2\|g\|_{BL}\mathbb{C}_{x,y}(\tau>N)
\end{split}
\end{align}
for any $g\in Lip_{b}(X)$, $x,y\in X$ and integers satisfying \hbox{$n>M>N>0$}, 
where
\begin{equation*}
\rho^{(N)}:=\inf\left\{n\geq N:\,\left(\phi_n^{(1)},\phi_n^{(2)}\right)\in F\;\;\text{and}\;\;V\left(\phi_n^{(1)}\right)+V\left(\phi_n^{(2)}\right)<\Gamma \right\}.
\end{equation*}
Then, conditon \hyperref[cnd:B2]{(B3)} allows one to estimate the first component on the right-hand side of \eqref{s1} by $\|g\|_{BL}\Gamma \delta^{n-M}$ with some $\delta\in (0,1)$. Hypotheses \hyperref[cnd:B1]{(B1)} and \hyperref[cnd:B5]{(B2)} are applied to show that there exists $q_{\rho}\in (0,1)$ and $C_{\rho}\in\mathbb{R}$ such that
\begin{equation}
\label{s2}
\mathbb{E}_{x,y}\left(q_{\rho}^{-\rho^{(N)}}\right)\leq C_{\rho}q_{\rho}^{-N}(1+V(x)+V(y))\quad\text{for any}\quad (x,y)\in X^2,
\end{equation} 
which, in particular, yields that $\mathbb{C}_{x,y}(\rho^{(N)}> M)\leq C_{\rho} q_{\rho}^{M-pN}(1+V(x)+V(y))$ for some $p\geq 1$. Further,  condition \hyperref[cnd:B3]{(B4)}, together with \hyperref[cnd:B2]{(B3)} and \hyperref[cnd:B4]{(B5)}, enables one to conclude that there exist $\epsilon>0$, $q_{\varkappa}\in (0,1)$ and $C_{\varkappa}<\infty$ such that
\begin{equation}
\label{s3}
\mathbb{C}_{x,y}(\tau=1)=\lim_{n\to\infty} Q^n(x,y,X^2)\geq \epsilon\quad\text{and}\;\;\; \mathbb{E}_{x,y}[\mathbbm{1}_{\{\varkappa<\infty\}}q_{\varkappa}^{-\varkappa}]\leq C_{\varkappa}\;\;\; \text{for any}\;\;\; (x,y)\in F,
\end{equation}
where $\varkappa:=\inf\{n\in\mathbb{N}:\,\theta_n=0\}$. Having established \eqref{s2} and \eqref{s3}, one can show (as in \hbox{\cite[Lemma 2.2]{ks}}) that there exist $q_\tau\in (0,1)$ and $C_{\tau}>0$ for which
\begin{equation*}
\widehat{\mathbb{E}}_{x,y}(q_{\tau}^{-\tau})\leq C_{\tau}(1+V(x)+V(y))\;\;\;\text{for any}\;\;\; (x,y)\in X^2,
\end{equation*}
whence $\mathbb{C}_{x,y}(\tau>N)\leq C_{\tau}q_{\tau}^N(1+V(x)+V(y))$. The assertion of Lemma \ref{cor:g_useful} then follows from \eqref{s1}.

A simple example of a model for which conditions \hyperref[cnd:B0]{(B0)}-\hyperref[cnd:B4]{(B5)} can be easily verified is a random iterated function system, considered e.g. in \cite{ks}. Note that 
within such an example 
the Markov operator $P$ and the substochastic kernel $Q$ are defined explicitly, and thus verifying conditions \hyperref[cnd:B0]{(B0)}-\hyperref[cnd:B4]{(B5)} becomes just a technical part.

\section{A Criterion on the Strassen Invariance Principle for the LIL}\label{sec:LIL}

The section is divided into two parts. The first one contains a few general observations concerning martingales defined on the path space of a given ergodic Markov chain, while the second one presents the main result of this paper, that is, a criterion on the Strassen invariance principle for the LIL for a class on non-stationary Markov-Feller chains. The proof techniques that we use are mainly based on~\cite{bms} and  \cite{hhsw2}.

Consider an arbitrary stochastic kernel $\Pi:X\times\mathcal{B}_X\to[0,1]$ and the corresponding operators $P:\mathcal{M}_{fin}(X)\to \mathcal{M}_{fin}(X)$  and $U: B_b(X)\to B_b(X)$, given by \eqref{def:markov_op} and \eqref{def:dual_op}, respectively. Further, fix $\mu\in\mathcal{M}_1(X)$, and let $(\phi_n)_{n\in\mathbb{N}_0}$ be an $X$-valued time-homogeneous Markov chain on a probability space $(\Omega, \mathcal{B}_{\Omega},\mathbb{P}_{\mu})$ with transition law $\Pi$ and initial distribution $\mu$. 

To streamline the forthcoming proofs, in what follows, we assume (without loss of generality) that $(\phi_n)_{n\in\mathbb{N}_0}$ is defined as a canonical chain on the coordinate space, and thus we take  $\Omega=X^{\mathbb{N}_0}$. By $(\mathcal{F}_n)_{n\in\mathbb{N}_0}$ we shall denote the natural filtration of this~chain. Moreover, we let $T:\Omega\to \Omega$ stand for the shift operator, that is, $T(x_0,x_1,\ldots):=(x_1,x_2,\ldots)$ for any $(x_0,x_1,\ldots)\in\Omega.$

\subsection{Auxiliary Results}\label{subsec}
In the remainder of this subsection, we assume that $P$ admits a unique invariant probability measure $\mu_*$, and that $(P^n\mu)_{n\in\mathbb{N}_0}$ converges weakly to $\mu_*$, as $n\to\infty$.

Let $(m_n)_{n\in\mathbb{N}_0}$  be a real-valued martingale with respect to $(\mathcal{F}_n)_{n\in\mathbb{N}_0}$ such that $z_n=z_1\circ T^{n-1}$ for any $n\in\mathbb{N}$, where $z_n:=m_n-m_{n-1}$ and $m_0:=0$. Further, define
\begin{equation}\label{def:sigma}
\sigma^2:=\mathbb{E}_{\mu_*}\left(z_1^2\right)\in(0,\infty), 
\end{equation}
and
\begin{equation*}
h_n^2(\mu):=\mathbb{E}_{\mu}\left(m_n^2\right)\;\;\;\mbox{for}\;\;\;n\in\mathbb{N}_0,\;\;\;\mu\in\mathcal{M}_1(X).
\end{equation*}

Now, let $\Sigma_{T}\subset\mathcal{F}$ denote the $\sigma$-algebra of $T$-invariant sets, i.e.
\begin{equation*}
\Sigma_{T}=\left\{A\in {\mathcal{F}}:\;1{\hskip -2.5 pt}\hbox{l}_{T^{-1}(A)}=\mathbbm{1}_A\;\;\mathbb{P}_{\mu_*}\mbox{-a.s.}\right\}. 
\end{equation*}
Since $\mu_*$ is the unique stationary distribution of $(\phi_n)_{n\in\mathbb{N}_0}$, it follows that the measure $\mathbb{P}_{\mu_*}$ is invariant and ergodic with respect to $T$ (see \cite[Proposition 7.16]{nonlinear}), that is,
\begin{equation*}
\mathbb{P}_{\mu_*}\left(T^{-1}(A)\right)=\mathbb{P}_{\mu_*}(A)\quad\mbox{for all}\quad A\in{\mathcal{F}} \quad \text{and} \quad P_{\mu_*}(A) \in \{0,1\}\quad \text{for any}\quad A\in\Sigma_T. 
\end{equation*}
The Birkhoff theorem for ergodic Markov chains (cf. \cite[Theorem 7.19]{nonlinear}) then implies the following statement:

\begin{lemma}\label{lem:birkhoff}
If\; $Z:\Omega\to \mathbb{R}$ is a $\mathbb{P}_{\mu_*}$-integrable random variable, then 
\begin{equation*}
\lim_{n\to\infty}\frac{1}{n}\sum_{l=0}^{n-1}Z\circ T^l=\mathbb{E}_{\mu_*}\left(Z|\Sigma_{T}\right)=\mathbb{E}_{\mu_*}(Z)\;\;\;\mathbb{P}_{\mu_*}\mbox{-a.s}.
\end{equation*} 
\end{lemma}

Let $(\phi_n^{(1)},\phi_n^{(2)})_{n\in\mathbb{N}_0}$ be an arbitrary Markovian coupling of $\Pi$ defined on some properly constructed probability space $(\bar{\Omega}, \bar{\mathcal{F}}, \mathbb{C})$, and recall that $\mathbb{E}_{x,y}$ denotes the expectation operator with respect to \hbox{$\mathbb{C}_{x,y}=\mathbb{C}(\cdot|\phi_0^{(1)}=x,\phi_0^{(2)}=y)$} for every $(x,y)\in X^2$.
For any given random variable $Z:\Omega\to \mathbb{R}$, we can now consider two copies of $Z$, defined as
\begin{equation}\label{def:copies}
Z^{(i)}(\omega):=Z\left(\phi_0^{(i)}(\omega),\phi_1^{(i)}(\omega),\ldots\right)\;\;\;\mbox{for}\;\;\;\omega\in\bar{\Omega}\;\;\;\mbox{and}\;\;\;i\in\{1,2\}.
\end{equation}

In what follows, we formulate a few lemmas, whose proofs are given in Section \hbox{\ref{append}.}

\begin{lemma}\label{lem:constant}
Suppose that   
\begin{equation}\label{cond:a}
\sum_{n=1}^{\infty}\mathbb{E}_{x,y}\left|z_n^{(1)}-z_n^{(2)}\right|<\infty\;\;\;\mbox{for any}\;\;\;(x,y)\in X^2.
\end{equation}
Then, for any $m\in\mathbb{N}\cup\{\infty\}$ and any $c\in\mathbb{R}_+$, the functions $f_{m,c}^{\inf}, f_{m,c}^{\sup}:X\to\mathbb{R}$, given by
\begin{align}\label{def:functions}
\begin{aligned}
f_{m,c}^{\inf}(x):=\mathbb{E}_x\left(\left|\liminf\limits_{n\to \infty}\frac{1}{n}\sum_{l=1}^n\left(z_l^2\wedge m\right)-c\right|\wedge 1\right),\\
f_{m,c}^{\sup}(x):=\mathbb{E}_x\left(\left|\limsup\limits_{n\to \infty}\frac{1}{n}\sum_{l=1}^n\left(z_l^2\wedge m\right)-c\right|\wedge 1\right)
\end{aligned}
\end{align}
are constant (and, in particular, continuous). By convention, we put $t\wedge m:=t$\, if\, $m=\infty$ and $t\in\mathbb{R}$.
\end{lemma}

\begin{lemma}\label{lem:continuous}
Suppose that the functions $f_{m,c}^{\inf}$ and $f_{m,c}^{\sup}$, given by \eqref{def:functions}, are continuous for all $m\in\mathbb{N}\cup\{\infty\}$ and any $c\in\mathbb{R}_+$. Then, for every $m\in\mathbb{N}\cup \{\infty\}$, we have
\begin{equation*}
\lim_{n\to\infty}\frac{1}{n}\sum_{l=1}^{n}\left(z_l^2\wedge m\right)=\mathbb{E}_{\mu_*}\left(z_1^2\wedge m\right)\;\;\;\mathbb{P}_{\mu}\mbox{-a.s.}
\end{equation*}
In particular, for $m=\infty$, we obtain
\begin{equation*}
\lim_{n\to\infty}\frac{1}{n}\sum_{l=1}^{n}z_l^2=\sigma^2\;\;\;\mathbb{P}_{\mu}\mbox{-a.s.},
\end{equation*}
where $\sigma^2$ is defined by \eqref{def:sigma}.
\end{lemma}

\begin{lemma}\label{lem:dawida}
Suppose that condition \eqref{cond:a} holds, and that, for some \hbox{$\mu\in\mathcal{M}_1(X)$,} there exists $r\in(0,2)$ such that
\begin{equation}\label{cond:b}
\sup_{n\in\mathbb{N}}\mathbb{E}_{\mu}
|z_n|^{2+r}<\infty.
\end{equation}
Then
\begin{equation}\label{eq:h_n_sigma}
\lim_{n\to\infty}\frac{h_n^2(\mu)}{n}=\sigma^2,
\end{equation}
and also
\begin{equation}\label{eq:h_n_1}
\lim_{n\to\infty}\frac{1}{h_n^2(\mu)}\sum_{l=1}^{n}z_l^2=1\;\;\;\mathbb{P}_{\mu}\mbox{-a.s.}
\end{equation}
\end{lemma}

\begin{lemma}\label{lem:ogolny}
Suppose that condition \eqref{cond:a} holds, and that  \eqref{cond:b} is fulfilled with some $r\in(0,2)$. Then, there exists $N\in\mathbb{N}$ such that $h_n(\mu)>0$ for all $n \geq N$, and the following statements hold:
\begin{align}
\sum_{n=N}^{\infty}
h_n^{-4}(\mu)\mathbb{E}_{\mu}\left(z_n^4
\mathbbm{1}_{\left\{|z_n|<\,\upsilon\, h_n(\mu)\right\}}\right)
< \infty\;\;\;\mbox{for every}\;\;\;\upsilon>0,\label{eq:lil1}\\
\sum_{n=N}^{\infty}
h_n^{-1}(\mu)\mathbb{E}\left(|z_n|
\mathbbm{1}_{\left\{|z_n|\geq\,\vartheta\, h_n(\mu)\right\}}\right)
<\infty\;\;\;\mbox{for every}\;\;\;\vartheta>0.\label{eq:lil2}
\end{align}
\end{lemma}

\subsection{The Invariance Principle for the LIL  for Certain Markov-Feller Chains}\label{sec:lil}
In the analysis that follows, we additionally require that the Markov operator $P$ enjoys the Feller property, stated as condition \hyperref[cnd:B0]{(B0)}, and that \hyperref[cnd:B1]{(B1)} holds with the Lyapunov function $V$ of the form
\begin{equation*}
V(x)=\rho(x,\bar{x}) \quad \text{for} \quad x\in X,
\end{equation*}
where $\bar{x}$ is an arbitrarily fixed point of $X$. Moreover, we assume that there exists a substochastic kernel $Q$ on $X^2 \times \mathcal{B}_{X^2}$, satisfying \eqref{def:substoch}, such that hypotheses \hyperref[cnd:B2]{(B2)}-\hyperref[cnd:B5]{(B5)} hold for some $F\subset X^2$. Under these settings, Theorem \ref{thm:spectral_gap} yields that $P$ possesses a unique invariant probability measure $\mu_*$, such that $\mu_* \in \mathcal{M}_{1,1}^V(X)$, and that condition \eqref{e:exp_erg} is fulfilled for some $q\in (0,1)$ and some $c<\infty$.

Let $g \in Lip_b(X)$, and define $\bar{g}:=g-\langle g,\mu_*\rangle$. Obviously, $\langle \bar{g},\mu_*\rangle=0$. Using \eqref{e:exp_erg}, for any $x\in X$ and any $i\in\mathbb{N}$, we can write
\begin{align}\label{new_estim}
\begin{aligned}
\left\langle \bar{g},P^i\delta_x\right\rangle
&=\left\langle \bar{g},P^i\delta_x\right\rangle-\left\langle\bar{g},\mu_*\right\rangle
\leq\|\bar{g}\|_{BL}\,d_{FM}\left(P^i\delta_x,\mu_*\right)\\
&\leq c\|\bar{g}\|_{BL}\,q^i\left(1+\varrho(x,\bar{x})+\left\langle \varrho(\cdot,\bar{x}),\mu_*\right\rangle\right)
\leq \hat{c}\|\bar{g}\|_{BL}\,q^i\left(1+\varrho(x,\bar{x})\right),
\end{aligned}
\end{align}
where $\hat{c}:=c(1+\langle \varrho(\cdot,\bar{x}),\mu_*\rangle)$. 
It then follows that
\begin{equation}\label{eq:integrable}
\sum_{i=0}^{\infty}\left|U^i\bar{g}(x)\right|
=\sum_{i=0}^{\infty}
\left|\left\langle \bar{g},P^i\delta_x-\mu_*\right\rangle\right|
\leq \frac{\hat{c}\|\bar{g}\|_{BL}}{1-{q}}\left(1+\varrho(x,\bar{x})\right),
\end{equation}
which enables us to define
\begin{equation}\label{def:chi}
\chi(\bar{g})(x):=\sum_{i=0}^{\infty}U^i\bar{g}(x)\;\;\;\mbox{for any}\;\;\; x\in X.
\end{equation}
Note that $\chi(\bar{g})$ has the following property:
\begin{align}\label{prop_chi}
\begin{aligned}
|\chi(\bar{g})(x)-\chi(\bar{g})(y)|
&\leq \sum_{i=0}^{\infty}\left|\left\langle \bar{g},P^i\delta_x-P^i\delta_y\right\rangle\right|
\leq \|\bar{g}\|_{BL}\sum_{i=0}^{\infty}d_{FM}\left(P^i\delta_x,P^i\delta_y\right)\\
&\leq \|\bar{g}\|_{BL}\sum_{i=0}^{\infty}\left(d_{FM}\left(P^i\delta_x,\mu_*\right)+d_{FM}\left(P^i\delta_y,\mu_*\right)\right)\\
&\leq \frac{{2\hat{c}\|\bar{g}\|_{BL}}}{1-{q}}\left(1+\varrho(x,\bar{x})+\varrho(y,\bar{x})\right)\quad \text{for any}\quad x,y\in X,
\end{aligned}
\end{align}
where the last inequality follows from \eqref{new_estim}. 

Now, introduce 
\begin{equation}\label{def:Mn}
M_0(\bar{g}):=0,\;\;\;
M_n(\bar{g}):=\chi(\bar{g})(\phi_n)-\chi(\bar{g})(\phi_0)+\sum_{i=0}^{n-1}\bar{g}(\phi_i)\;\;\;\mbox{for}\;\;\;n\in\mathbb{N},
\end{equation}
and observe that $(M_n(\bar{g}))_{n\in\mathbb{N}_0}$
is a martingale with respect to the natural filtration of $(\phi_n)_{n\in\mathbb{N}_0}$ (for the proof see e.g. \cite[Lemma 3]{hhsw2}). Furthermore, we define 
\begin{gather}
\label{def:Zn}
Z_n(\bar{g}):=M_n(\bar{g})-M_{n-1}(\bar{g})=\chi(\bar{g})(\phi_n)-\chi(\bar{g})(\phi_{n-1})+\bar{g}(\phi_{n-1}), \;\;\;n\in\mathbb{N},\\
\label{def:sigmaLIL}
\sigma^2(\bar{g}):=\mathbb{E}_{\mu_*}\left(Z_1^2(\bar{g})\right),\\
\label{def:h_n}
h_n^2(\mu)(\bar{g}):=\mathbb{E}_{\mu}\left(M_n^2(\bar{g})\right),\;\;\; n\in\mathbb{N}_0.
\end{gather}
One can easily check that $Z_n(\bar{g})=Z_1(\bar{g})\circ T^{n-1}$ for any $n\in\mathbb{N}$. 

Let us now define $\mathcal{C}$ as a Banach space of all real-valued continuous functions on $[0,1]$ with the supremum norm. By $\mathcal{K}$ we will denote the subspace of $\mathcal{C}$ consisting of all absolutely continuous functions~$f$ such that $f(0)=0$ and $\int_0^1\left(f'(t)\right)^2dt\leq 1$. Further, consider the sequence of random variables $(r_n(\bar{g}))_{n\in\mathbb{N}_0}$ with values in~$\mathcal{C}$, determined by

\begin{align}\label{def:r_n}
\begin{aligned}
&r_n(\bar{g})(t):=
\frac{\sum_{i=0}^{k-1}\bar{g}(\phi_i)+(nt-k)\bar{g}(\phi_{k})}{\sigma(\bar{g})\sqrt{2n\ln\ln n}}
\;\;\;\mbox{for}\;\;\;n>e,\;t\in(0,1]\\
&\hspace{1.83cm}\text{and}\;k\in\{1,\ldots,n-1\}\;\;\text{such that}\;\;k\leq nt\leq k+1,\\
&r_n(\bar{g})(t):=0\;\;\;\mbox{for}\;\;\; n\leq e \;\;\;\mbox{or} \;\;\;t=0.
\end{aligned}
\end{align}

For any~given function $g\in Lip_b(X)$, we say that  the Markov chain $\left(g(\phi_n)\right)_{n\in\mathbb{N}_0}$ satisfies the invariance principle for the LIL if $0<\sigma^2(\bar{g})<\infty$, the family \hbox{$\{r_n(\bar{g}):\,{n\in\mathbb{N}_0}\}$} is relatively compact in $\mathcal{C}$, and the set of its limit points coincides with~$\mathcal{K}$ $\mathbb{P}_{\mu}$-a.s. Observe that, whenever the chain $\left({g}(\phi_n)\right)_{n\in\mathbb{N}_0}$ satisfies the invariance principle for the LIL, it also obeys the LIL itself. Indeed, if $0<\sigma^2(\bar{g})<\infty$, then for any $n>e$ we can define
\begin{equation*}
\hat{r}_n(\bar{g}):=r_n(\bar{g})(1)=\frac{\sum_{i=1}^{n}\bar{g}(\phi_{i})}{\sigma(\bar{g})\sqrt{2n\ln\ln n}},
\end{equation*}
which, due to the definition of $\mathcal{K}$,  satisfies
\begin{equation*}
\limsup\limits_{n\to\infty}\hat{r}_n\left(\bar{g}\right)=1
\;\;\;
\mbox{and}\;\;\;
\liminf\limits_{n\to\infty}\hat{r}_n\left(\bar{g}\right)=-1\;\;\;\mathbb{P}_{\mu}\mbox{-a.s}.
\end{equation*}

Our aim now is to establish the main result of this paper. It shall be  formulated in the same spirit as Theorem \ref{thm:spectral_gap} and \cite[Theorem 3.2]{clt_chw} (see \cite{dawid2, dawid,clt_chw, asia} for possible applications of these theorems). While hypotheses \hyperref[cnd:B0]{(B0)}-\hyperref[cnd:B4]{(B5)} are sufficient for the Markov operator $P$ to be exponentially ergodic in $d_{FM}$, our proof of the Strassen invariance principle for the LIL will additionally require the following condition:
\begin{itemize}
\item[(B1$^*$)]\phantomsection \label{cnd:B1p} There exist $a^*\in(0,1)$ and $b^{*}\in(0,\infty)$ such that, for any $\nu\in \mathcal{M}_{1,2+r}^{\varrho(\cdot,\bar{x})}(X)$,
\begin{equation*}
\left\langle \varrho^{2+r}(\cdot,\bar{x}),\,P\nu\right\rangle^{1/(2+r)}
\leq a^{*}\left\langle \varrho^{2+r}(\cdot,\bar{x}),\,\nu\right\rangle^{1/(2+r)}+b^{*}.
\end{equation*}
\end{itemize}
\begin{remark}
Let us compare condition \hyperref[cnd:B1p]{(B1$^{*}$)} with (B1$^{\prime}$), which has been employed in~\cite{clt_chw} to establish the CLT for a subclass of Markov-Feller chains described in Section \ref{sec:SGP}. The latter guarantees the existence of $a\in(0,1)$ and $b\in (0,\infty)$ such that
\begin{equation*}
U \varrho^2(x,\bar{x})\leq (a\,\varrho(x,\bar{x})+b)^2\quad\text{for any}\quad x\in X.
\end{equation*}
One can observe that (B1$^{\prime}$) is a stronger version of \hyperref[cnd:B1]{(B1)}. Condition \hyperref[cnd:B1p]{(B1$^{*}$)} is of the same type, although, in general, it does not need to imply \hyperref[cnd:B1]{(B1)}. Consequently, in Theorem \ref{thm:LIL} we assume both \hyperref[cnd:B1]{(B1)} and~\hyperref[cnd:B1p]{(B1$^{*}$)}. 
\end{remark}

\begin{theorem}\label{thm:LIL}
Suppose that $(\phi_n)_{n\in\mathbb{N}_0}$ is an $X$-valued time-homogeneous Markov chain with transition law $\Pi$ and initial distribution $\mu$ such that $\mu\in\mathcal{M}_{1,2+r}^{\varrho(\cdot,\bar{x})}(X)$ for some  $r\in(0,2)$. Let $P$ denote the Markov operator corresponding to $\Pi$. Further, assume that there exists a~substochastic kernel \hbox{$Q:X^2\times\mathcal{B}_{X^2}\to[0,1]$} satisfying (\ref{def:substoch}), such that conditions \hyperref[cnd:B0]{(B0)}-\hyperref[cnd:B4]{(B5)} and \hyperref[cnd:B1p]{(B1$^{*}$)} hold for $P$ and $Q$ with some $F\subset X^2$. Then, for any $g\in Lip_b(X)$, the constant $\sigma^2(\bar{g})$, defined by \eqref{def:sigmaLIL}, can be written as
\begin{equation}
\label{dod2}
\sigma^2(\bar{g})=\<\chi^2(\bar{g})-\left(U\chi(\bar{g})\right)^2,\mu_*\>,
\end{equation}
where $\mu_*$ stands for the unique invariant measure of $P$, and the chain $\left({g}(\phi_n)\right)_{n\in\mathbb{N}_0}$ obeys the Strassen invariance principle for the LIL, whenever $\sigma^2(\bar{g})>0$. 

\end{theorem}

Before we prove Theorem \ref{thm:LIL}, we first need to state several auxiliary facts.  
Lemmas~\ref{lem:1_}-\ref{lem:2}, established below, concern certain properties of $(Z_n(\bar{g}))_{n\in\mathbb{N}_0}$, given by \eqref{def:Zn}, while Lemma \ref{lem:3} indicates mutual relations between $\sigma^2(\bar{g})$ and $h_n^2(\mu)(\bar{g})$, given by \eqref{def:sigmaLIL} and \eqref{def:h_n}, respectively. Finally, Lemmas \ref{lem:5} and \ref{lem:6} allow us to assure a form of the functional LIL for the sequence $(Z_n(\bar{g}))_{n\in\mathbb{N}}$ of martingale increments (cf. \cite[Theorem 1]{hs}).

Let $(\phi^{(1)}, \phi^{(2)})_{n\in\mathbb{N}_0}$ be a coupling of $\Pi$ such that condition \hyperref[cnd:B4]{(B5)} holds.
\begin{lemma}\label{lem:1_}
Under the assumptions of Theorem \ref{thm:LIL}, we have
\begin{equation*}
\sum_{n=1}^{\infty}\mathbb{E}_{x,y}\left|Z_n^{(1)}(\bar{g})-Z_n^{(2)}(\bar{g})\right|<\infty \;\;\;\mbox{for}\;\;\;x,y\in X,
\end{equation*}
where  $Z_n^{(1)}$ and $Z_n^{(2)}$ are defined according to the rule given in \eqref{def:copies}, applied to the above-specified coupling $(\phi^{(1)}, \phi^{(2)})_{n\in\mathbb{N}_0}$.
\end{lemma}

\begin{proof}
First of all, note that
\begin{align}
\begin{split}
\label{eq:estim_difference}
\left|Z_n^{(1)}(\bar{g})-Z_n^{(2)}(\bar{g})\right|
&\leq\left|\chi(\bar{g})\left(\phi_n^{(1)}\right)-\chi(\bar{g})\left(\phi_n^{(2)}\right)\right|
+\left|\chi(\bar{g})\left(\phi_{n-1}^{(1)}\right)-\chi(\bar{g})\left(\phi_{n-1}^{(2)}\right)\right|\\
&\quad+2\|\bar{g}\|_{\infty}.
\end{split}
\end{align}
Further, we can deduce that
\begin{align*}
\mathbb{E}_{x,y}\Big|\chi(\bar{g})\left(\phi_n^{(1)}\right)-&\chi(\bar{g})\left(\phi_n^{(2)}\right)\Big|
\leq \sum_{i=0}^{\infty}\mathbb{E}_{x,y}\left|U^i\bar{g}\left(\phi_n^{(1)}\right)-U^i\bar{g}\left(\phi_n^{(2)}\right)\right|\\
&= \sum_{i=0}^{\infty}\int_{X^2}\left|U^i\bar{g}(u_1)-U^i\bar{g}(u_2)\right|\,C^n(x,y,du_1\times du_2)\\
&= \sum_{i=0}^{\infty}\int_{X^2}\left|\left\langle\bar{g},P^i\delta_{u_1}\right\rangle-\left\langle\bar{g},P^i\delta_{u_2}\right\rangle\right|\,C^n(x,y,du_1\times du_2)\\
&\leq \sum_{i=0}^{\infty}\int_{X^2}\int_{X^2}\left|g(v_1)-g(v_2)\right|C^i(u_1,u_2,dv_1\times dv_2)\,C^n(x,y,du_1\times du_2)\\
&= \sum_{i=0}^{\infty}\int_{X^2}
\left|g(v_1)-g(v_2)\right|C^{n+i}(x,y,dv_1\times dv_2)\\
&=\sum_{i=n}^{\infty}\mathbb{E}_{x,y}\left|g\left(\phi_i^{(1)}\right)-g\left(\phi_i^{(2)}\right)\right|.
\end{align*}
Hence, applying Lemma \ref{cor:g_useful}, we infer that there exist $q\in(0,1)$ and $c<\infty$ such that 
\begin{align}\label{eq:estim_chi}
\begin{aligned}
\mathbb{E}_{x,y}\left|\chi(\bar{g})\left(\phi_n^{(1)}\right)-\chi(\bar{g})\left(\phi_n^{(2)}\right)\right|
&\leq
c\|\bar{g}\|_{BL}\left(1+\varrho(x,\bar{x})+\varrho(y,\bar{x})\right)\sum_{i=n}^{\infty} q^i\\
&\leq c\|\bar{g}\|_{BL}\,q^n(1-q)^{-1}\left(1+\varrho(x,\bar{x})+\varrho(y,\bar{x})\right)
\end{aligned}
\end{align}
for every $n\in\mathbb{N}$.  
Combining \eqref{eq:estim_difference} with \eqref{eq:estim_chi}, finally gives  
\begin{equation*}
\sum_{n=1}^{\infty}\mathbb{E}_{x,y}\left|Z_n^{(1)}(\bar{g})-Z_n^{(2)}(\bar{g})\right|<\infty,
\end{equation*}
which completes the proof.
\end{proof}

\begin{lemma}\label{lem:2_}
Under the assumptions of Theorem \ref{thm:LIL}, we have 
\begin{equation*}
\sup_{n\in\mathbb{N}}\mathbb{E}_{\mu}\left|Z_n(\bar{g})\right|^{2+r}<\infty \quad \text{for any}\quad g\in Lip_b(X).
\end{equation*}
\end{lemma}
\begin{proof}
Let $g\in Lip_b(X)$ and $n\in\mathbb{N}$. One can easily check that, for every $\kappa>0$, there exists $p\in(2,\infty)$ such that 
\begin{equation}\label{eq:abbrev_formula}
(s+t)^{2+\kappa}\leq p\left(s^{2+\kappa}+t^{2+\kappa}\right)\;\;\;\mbox{for any}\;\;\;s,t\geq 0.
\end{equation}
Hence, due to the definition of $Z_n(\bar{g})$,  
we obtain
\begin{align*}
\mathbb{E}_{\mu}\left|Z_n(\bar{g})\right|^{2+r}
&\leq\, p\int_X \left|\chi(\bar{g})(u)\right|^{2+r}\,P^{n}\mu(du)
+p^2\int_X \left|\chi(\bar{g})(u)\right|^{2+r}\,P^{n-1}\mu(du)\\
&\quad+p^2\int_X\left|\bar{g}(u)\right|^{2+r}\,P^{n-1}\mu(du),
\end{align*}
where the last term can be majorized by $p^2\|\bar{g}\|_{\infty}^{2+r}$.  
Then, according to \eqref{prop_chi}, there exist $q\in(0,1)$ and $\hat{c}>0$ such that, for all $n\in\mathbb{N}$,
\begin{align*}
\int_X \left|\chi(\bar{g})(u)\right|^{2+r}P^n\mu(du)
&=\int_X\left|\chi(\bar{g})(u)-\chi(\bar{g})(\bar{x})+\chi(\bar{g})(\bar{x})\right|^{2+r}P^n\mu(du)\\
&\leq p\left|\chi(\bar{g})(\bar{x})\right|^{2+r}+p\int_X\left|\chi(\bar{g})(u)-\chi(\bar{g})(\bar{x})\right|^{2+r}P^n\mu(du)\\
&\leq p\left|\chi(\bar{g})(\bar{x})\right|^{2+r}+p^2\left(\frac{2\hat{c}\|\bar{g}\|_{BL}}{1-q}\right)^{2+r}\left(1+\left\langle \varrho^{2+r}(\cdot,\bar{x}),P^n\mu\right\rangle\right). 
\end{align*} 
Further, from \hyperref[cnd:B1p]{(B1$^{*}$)}, it follows that
\begin{align}\label{estim_}
\begin{aligned}
\left\langle \varrho^{2+r}(\cdot,\bar{x}),P^n\mu\right\rangle^{1/(2+r)}
&\leq a^{*}\left\langle \varrho^{2+r}(\cdot,\bar{x}),P^{n-1}\mu\right\rangle^{1/(2+r)}+b^{*}\\
&\leq\ldots
\leq\left(a^{*}\right)^n\left\langle\varrho^{2+r}(\cdot,\bar{x}),\mu\right\rangle^{1/(2+r)}+\frac{b^{*}}{1-a^{*}},
\end{aligned}
\end{align}
which gives
\begin{equation*}
\left\langle \varrho^{2+r}(\cdot,\bar{x}),P^n\mu\right\rangle\leq\left(\left\langle\varrho^{2+r}(\cdot,\bar{x}),\mu\right\rangle^{1/(2+r)}+\frac{b^{*}}{1-a^{*}}\right)^{2+r}\;\;\;\mbox{for all}\;\;\;n\in\mathbb{N}.
\end{equation*}
Finally, recalling that $\mu\in\mathcal{M}_{1,2+r}^{\varrho(\cdot,\bar{x})}(X)$, we obtain 
\begin{equation}\label{eq:sup_nE_x}
\sup_{n\in\mathbb{N}}\mathbb{E}_{{\mu}}\left|Z_n(\bar{g})\right|^{2+r}< p\bar{c}+p^2\bar{c}+p^2\|\bar{g}\|_{\infty}^{2+r}
\end{equation}
with
\begin{equation*}
\bar{c}=p\left|\chi(\bar{g})(\bar{x})\right|^{2+r}+p^2\left(\frac{2\hat{c}\|\bar{g}\|_{BL}}{1-q}\right)^{2+r}
\left(1+\left(\left\langle\varrho^{2+r}(\cdot,\bar{x}),\mu\right\rangle^{1/(2+r)}+\frac{b^{*}}{1-a^{*}}\right)^{2+r}\right).
\end{equation*}
The proof of Lemma \ref{lem:2_} is therefore completed.
\end{proof}

\begin{lemma}\label{lem:2}
Under the assumptions of Theorem \ref{thm:LIL}, we have 
\begin{equation*}
\sigma^2(\bar{g})=\mathbb{E}_{\mu_*}Z_1^2(\bar{g})<\infty \quad \text{for any}\quad g\in Lip_b(X).
\end{equation*}
\end{lemma}

\begin{proof}
For every $k\in\mathbb{N}$, we define $\tilde{V}_k:X\to[0,k]$ by $\tilde{V}_k(x)=\min\{k,\varrho^{2+r}(x,\bar{x})\}$ for $x\in X$. 
Note that $\tilde{V}_k\in C_b(X)$ for all $k\in\mathbb{N}$. Hence, keeping in mind that $\mu\in\mathcal{M}_{1,2+r}^{\varrho(\cdot,\bar{x})}(X)$ and that $P^n \mu$ converges weakly to~$\mu_*$, as $n\to\infty$, we have
\begin{equation}\label{eq:thm5.1_}
\langle \tilde{V}_k,\mu_*\rangle=\lim_{n\to\infty}\langle \tilde{V}_k,P^n\mu\rangle\quad\mbox{for every}\quad k\in\mathbb{N}.
\end{equation}
Observe that $(\tilde{V}_k)_{k\in\mathbb{N}}$ is a non-increasing sequence of non-negative functions satisfying \linebreak\hbox{$\lim_{k\to\infty}\tilde{V}_k(x)=\varrho^{2+r}(x,\bar{x})$} for any $x\in X$. Therefore, using the Monotone Convergence Theorem, together with (\ref{eq:thm5.1_}) and (\ref{estim_}),  we obtain
\begin{align*}
\left\langle \varrho^{2+r}(\cdot,\bar{x}),\mu_*\right\rangle
=\lim_{k\to\infty}\left\langle \tilde{V}_k,\mu_*\right\rangle
&=\lim_{k\to\infty}\lim_{n\to\infty}\left\langle \tilde{V}_k,P^n\mu\right\rangle\\
&\leq\limsup\limits_{n\to\infty}\left\langle \varrho^{2+r}(\cdot,\bar{x}),P^n\mu\right\rangle
\leq\left(\frac{b^{*}}{1-a^{*}}\right)^{2+r},
\end{align*}
which implies that $\mu_*\in \mathcal{M}_{1,2+r}^{\varrho(\cdot,\bar{x})}(X)$.

Hence, according to Lemma \ref{lem:2_} and the H\"older inequality, we in particular obtain 
$\mathbb{E}_{\mu_*}Z_1^2(\bar{g})<\infty$, 
which completes the proof.
\end{proof}

\begin{lemma}\label{lem:3}
Under the assumptions of Theorem \ref{thm:LIL}, we have 
\begin{equation*}
\lim_{n\to\infty}\frac{h^2_n(\mu)(\bar{g})}{n}=\sigma^2(\bar{g}),
\end{equation*}
where $\sigma(\bar{g})$ and $h_n(\mu)(\bar{g})$ are defined by \eqref{def:sigmaLIL} and \eqref{def:h_n}, respectively.
\end{lemma}

\begin{proof}
The assertion follows from Lemma \ref{lem:dawida}. Note that conditions \eqref{cond:a} and \eqref{cond:b} are provided by Lemmas \ref{lem:1_} and \ref{lem:2_}, respectively.
\end{proof}

\begin{lemma}\label{lem:5}
Under the assumptions of Theorem \ref{thm:LIL}, $(Z_n(\bar{g}))_{n\in\mathbb{N}}$ and $(h_n(\mu)(\bar{g}))_{n\in\mathbb{N}}$, given by \eqref{def:Zn} and \eqref{def:h_n}, respectively, are related with each other in the  following way:
\begin{equation}\label{cond:3}
\lim_{n\to\infty}\frac{1}{h_n^2(\mu)(\bar{g})}\sum_{l=1}^nZ_l^2(\bar{g})=1\;\;\;\mathbb{P}_{\mu}\mbox{-a.s.}
\end{equation}
\end{lemma}

\begin{proof}
Lemmas \ref{lem:1_} and \ref{lem:2_} guarantee that $(Z_n(\bar{g}))_{n\in\mathbb{N}}$ satisfies the assumptions of Lemma \ref{lem:dawida}, which in turn  implies the assertion of this lemma.
\end{proof}

\begin{lemma}\label{lem:6}
Under the assumptions of Theorem \ref{thm:LIL}, we have
\begin{align}
\label{cond:1}
&\sum_{n=1}^{\infty}
h_n^{-4}(\mu)(\bar{g})\mathbb{E}\left(Z_n^4(\bar{g})
\mathbbm{1}_{\left\{|Z_n(\bar{g})|<\,\upsilon \,h_n(\mu)(\bar{g})\right\}}\right)
< \infty\;\;\;\mbox{for every}\;\;\;\upsilon>0,\\
\label{cond:2}
&\sum_{n=1}^{\infty}
h_n^{-1}(\mu)(\bar{g})\mathbb{E}\left(|Z_n(\bar{g})|
\mathbbm{1}_{\left\{|Z_n(\bar{g})|\geq\,\vartheta \,h_n(\mu)(\bar{g})\right\}}\right)
<\infty\;\;\;\mbox{for every}\;\;\;\vartheta>0.
\end{align}
\end{lemma}

\begin{proof}
Having in mind that condition \eqref{cond:b} is provided by Lemma \ref{lem:2_}, we see that the claim follows directly from Lemma \ref{lem:ogolny}.
\end{proof}
We are now in a position to prove the main theorem. As mentioned earlier, an essential step in our proof will be applying \cite[Theorem 1]{hs}.
\begin{proof}[Proof of Theorem \ref{thm:LIL}]
The existence and uniqueness of an invariant probability measure for~$P$, further denoted by $\mu_*$, follows from Theorem \ref{thm:spectral_gap}. The proof proceeds in two steps.

\textbf{Step I. }
First of all, we will show that the sequence $(h_n(\mu)(\bar{g}))_{n\geq{N}}$, where $h_n(\mu)(\bar{g})$ is given by \eqref{def:h_n}, is strictly increasing for some sufficiently large $N\in\mathbb{N}$, which equivalently means that $\mathbb{E}_{\mu}(Z_n^2(\bar{g}))>0$ for any $n\geq N$. Along the way, we will also get \eqref{dod2}. Since \hbox{$Z_n(\bar{g})=Z_1(\bar{g})\circ T^{n-1}$}, we can write
\begin{align}\label{eq:main_proof}
\begin{aligned}
\mathbb{E}_{\mu}\left(Z_n^2(\bar{g})\right)
&=\mathbb{E}_{\mu}\left(\mathbb{E}_{\mu}\left(Z_1^2(\bar{g})\circ T^{n-1}|\mathcal{F}_{n-1}\right)\right)
=\mathbb{E}_{\mu}\left(\mathbb{E}_{\phi_{n-1}}\left(Z_1^2(\bar{g})\right)\right)\\
&=\int_{\Omega}\mathbb{E}_{\phi_{n-1}(\omega)}\left(Z_1^2(\bar{g})\right)\mathbb{P}_{\mu}(d\omega)
=\int_X\mathbb{E}_x\left(Z_1^2(\bar{g})\right)P^{n-1}\mu(dx),
\end{aligned}
\end{align}
where the second equality follows from the Markov property. 
According to \eqref{def:Zn}, we have
\begin{align}\label{eq:E_xZ_1^2}
\begin{aligned}
\mathbb{E}_x\left(Z_1^2(\bar{g})\right)&=\mathbb{E}_x\left(\left(\chi(\bar{g})(\phi_1)-\chi(\bar{g})(\phi_0)+\bar{g}(\phi_0)\right)^2\right)\\
&=U\chi^2(\bar{g})(x)+\chi^2(\bar{g})(x)+\bar{g}^2(x)+2\,\bar{g}(x)\,U\chi(\bar{g})(x)\\
&\quad-2\,\chi(\bar{g})(x)\,U\chi(\bar{g})(x)
-2\,\chi(\bar{g})(x)\,\bar{g}(x)\quad\text{for any}\quad x\in X.
\end{aligned}
\end{align}
Note that $\chi^2(\bar{g})\in\bar{B}_b(X)$, and therefore we can apply to it the extension of the dual operator~$U$, given by \eqref{def:dual_op}. On the other hand, from \eqref{eq:integrable} and \hyperref[cnd:B1]{(B1)} (with \hbox{$V(\cdot)=\varrho(\cdot,\bar{x})$}) it follows that $\chi(\bar{g})$ is integrable with respect to $P\delta_x$ for every $x\in X$, and thus $U\chi(\bar{g})(x)=\langle \chi(\bar{g}), P\delta_x\rangle$ is well-defined for any $x\in X$. 
Further, observe that, for every $x\in X$,
\begin{align*}
U\chi(\bar{g})(x)&=\int_X\sum_{i=0}^{\infty}U^i\bar{g}(y)\,P\delta_x(dy)=\sum_{i=0}^{\infty}\int_XU^{i+1}\bar{g}(y)\,\delta_x(dy)=\sum_{i=1}^{\infty}U^i\bar{g}(x)=\chi(\bar{g})(x)-\bar{g}(x).
\end{align*}
Combining this with \eqref{eq:E_xZ_1^2} gives
\begin{align}\label{eq:}
\begin{aligned}
\mathbb{E}_x\left(Z_1^2(\bar{g})\right)
&=U\chi^2(\bar{g})(x)+\chi^2(\bar{g})(x)+\bar{g}^2(x)+2\,\bar{g}(x)\left(\chi(\bar{g})(x)-\bar{g}(x)\right)\\
&\quad-2\,\chi(\bar{g})(x)\left(\chi(\bar{g})(x)-\bar{g}(x)\right)-2\,\chi(\bar{g})(x)\,\bar{g}(x)\\
&=U\chi^2(\bar{g})(x)-\left(\chi(\bar{g})(x)-\bar{g}(x)\right)^2
=U\chi^2(\bar{g})(x)-\left(U\chi(\bar{g})\right)^2(x).
\end{aligned}
\end{align}
Hence, keeping in mind \eqref{eq:main_proof}, we see that
\begin{equation}
\label{dod}
\mathbb{E}_{\mu}\left(Z_n^2(\bar{g})\right)=\int_X \left[U\chi^2(\bar{g})(x)-\left(U\chi(\bar{g})\right)^2(x)\right]P^{n-1}\mu(dx),
\end{equation}
which, in particular, proves \eqref{dod2}, since
$$\sigma^2(\bar{g})=\mathbb{E}_{\mu_*}\left(Z_1^2(\bar{g})\right)= \<U\chi^2(\bar{g})-\left(U\chi(\bar{g})\right)^2,\mu_*\>=\<\chi^2(\bar{g})-\left(U\chi(\bar{g})\right)^2,\mu_*\>.$$
Now, observe that $U\chi^2(\bar{g})-\left(U\chi(\bar{g})\right)^2$ is non-negative (by the Cauchy--Schwarz inequality), and it is also continuous, due to the continuity of $\chi(\bar{g})$ and the fact that $P$ is Feller. Obviously, the former can be justified by using the Weierstrass M-test, together with \eqref{eq:integrable}. Consequently, taking into account that $P^n\mu\stackrel{w}{\to}\mu_*$ (which holds by Theorem~\ref{thm:spectral_gap}), we can apply the Portmanteau to \eqref{dod}, which yields that
$$\liminf_{n\to\infty}\mathbb{E}_{\mu}\left(Z_n^2(\bar{g})\right)\geq \<U\chi^2(\bar{g})-\left(U\chi(\bar{g})\right)^2,\mu_*\>=\<\chi^2(\bar{g})-\left(U\chi(\bar{g})\right)^2,\mu_*\>>0.$$
This finally implies that the sequence $(h_n(\mu))_{n\geq N}$ is strictly increasing for some sufficiently large \hbox{$N\in\mathbb{N}$}, as claimed.

In view of the above, we may assume, without loss of generality, that $(h_n(\mu)(\bar{g}))_{n\in\mathbb{N}_0}$ is strictly increasing, and therefore we are allowed to introduce 
\begin{align}\label{def:eta}
\begin{aligned}
&\eta_n(\bar{g})(t):=\frac{M_k(\bar{g})+\cfrac{h_n^2(\mu)(\bar{g})\,t-h_k^2(\mu)(\bar{g})}{h_{k+1}^2(\mu)(\bar{g})-h_k^2(\mu)(\bar{g})}\,Z_{k+1}(\bar{g})}{\sigma(\bar{g})\sqrt{2n\ln\ln n}}\;\;\;
\mbox{for}\;\; n>e,\;\;t\in(0,1]\\
&\hspace{1.9cm}\mbox{and}\;\;1\leq k\leq n-1\;\;\text{such that}\;\; h_k^2(\mu)(\bar{g})\leq h_n^2(\mu)(\bar{g})\,t\leq h_{k+1}^2(\mu)(\bar{g}),\; \\
&\eta_n(\bar{g})(t):=0\;\;\;\mbox{for}\;\;\; n\leq e\;\;\;\mbox{or}\;\;\;t=0.
\end{aligned}
\end{align}
Note that, according to Lemma \ref{lem:3}, we have
\begin{equation}\label{limit}
\lim_{n\to\infty}\frac{\sqrt{2\,h_n^2(\mu)(\bar{g})\,\ln\ln h_n^2(\mu)(\bar{g})}}{\sigma(\bar{g})\sqrt{2n\ln \ln n}}=1.
\end{equation}
Further, Lemmas \ref{lem:5} and \ref{lem:6} ensure conditions \eqref{cond:3} and \eqref{cond:1}, \eqref{cond:2}, respectively. 
Combining this with \eqref{limit} and referring to \cite[Theorem 1]{hs}, we can conclude that 
$\{\eta_n(\bar{g}):\,n\in\mathbb{N}_0\}$ is relatively compact in $\mathcal{C}$, and that the set of its limit points coincides with~$\mathcal{K}$ $\mathbb{P}_{\mu}$-a.s.

Now, let us define
\begin{align}
\begin{split}
\label{def:tilde_eta}
&\tilde{\eta}_n(\bar{g})(t)
=:\frac{M_k(\bar{g})-(nt-k)\,Z_{k+1}(\bar{g})}{\sigma(\bar{g})\sqrt{2n\ln\ln n}}
\;\;\;
\mbox{for}\;\;\; n>e\;\;\text{and}\;\; t\in(0,1],\\
&\hspace{1.76cm}\text{whenever}\;\;k\in\{1,\ldots,n-1\}\;\;\text{is such that}\;\;k\leq nt\leq k+1,\\
&\tilde{\eta}_n(\bar{g})(t):=0
\;\;\;\mbox{for}\;\;\; n\leq e\;\;\;\mbox{or}\;\;\;t=0.
\end{split}
\end{align}

Observe that, for any $t\in (0,1]$, $n>e$ and $k\in\{1,\ldots,n-1\}$ satisfying $k\leq nt\leq k+1$, one can write
\begin{equation}\label{eq:norming}
\frac{k\sigma^2(\bar{g})}{h_k^2(\mu)(\bar{g})}\,h_k^2(\mu)(\bar{g})
\leq \frac{n\sigma^2(\bar{g})}{h_n^2(\mu)(\bar{g})}\,t\,h_n^2(\mu)(\bar{g})
\leq\frac{(k+1)\sigma^2(\bar{g})}{h_{k+1}^2(\mu)(\bar{g})}\,h_{k+1}^2(\mu)(\bar{g}).
\end{equation}
By virtue of Lemma \ref{lem:3} we know that $\lim_{n\to\infty}n\sigma^2(\bar{g})h_n^{-2}(\mu)(\bar{g})=1$. Hence, due to \eqref{eq:norming}, for any $\epsilon>0$ and sufficiently large $n$, we obtain
\begin{equation}
\label{estim_lil}
\frac{1-{\epsilon}}{1+{\epsilon}}\,h_k^2(\mu)(\bar{g})\leq t\,h_n^2(\mu)(\bar{g})\leq\frac{1+{\epsilon}}{1-{\epsilon}}\,h_{k+1}^2(\mu)(\bar{g})\;\;\;\text{for every}\;\;\;k\in\mathbb{N}\;\;\;\text{such that}\;\;\;k\leq nt\leq k+1.
\end{equation}

Our goal for now is to prove that the functional LIL holds for the sequence $\left(Z_n(\bar{g})\right)_{n\in\mathbb{N}}$, that is, \hbox{$\{\tilde{\eta}_n(\bar{g}):n\in\mathbb{N}_0\}$} is relatively compact in $\mathcal{C}$, and the set of its limit points coincides with~$\mathcal{K}$ $\mathbb{P}_{\mu}$-a.s. For this purpose, it suffices to show that, for every $t\in(0,1]$, there exists a sequence $(t_n)_{n\in\mathbb{N}}$ of positive numbers such that
\begin{equation*}
\tilde{\eta}_n(\bar{g})(t)=\eta_n(\bar{g})(t_n)\quad\mbox{for every}\quad n\in\mathbb{N}\quad\text{and}\quad \lim_{n\to\infty}t_n=t.
\end{equation*}
To do this, fix $n>e$, and let $k$ be such that $k\leq nt\leq k+1$. According to definitions \eqref{def:eta} and~\eqref{def:tilde_eta}, we see that the equality $\tilde{\eta}_n(\bar{g})(t)=\eta_n(\bar{g})(t_n)$ is satisfied for
\begin{equation}\label{def:t_n}
t_n:=\frac{(nt-k)\left(h_{k+1}^2(\mu)(\bar{g})-h_k^2(\mu)(\bar{g})\right)+h_k^2(\mu)(\bar{g})}{h_n^2(\mu)(\bar{g})},
\end{equation}
whenever 
\begin{equation}\label{to_be_obeyed}
h_k^2(\mu)(\bar{g})\leq t_n\,h_n^2(\mu)(\bar{g})\leq h_{k+1}^2(\mu)(\bar{g}).
\end{equation}
But \eqref{to_be_obeyed} obviously holds,  since $0\leq nt-k \leq 1$. 
Moreover, for every ${\epsilon}$ and sufficiently large $n$, we have
\begin{equation*}
t_n\in\left[t\,\frac{1-{\epsilon}}{1+{\epsilon}},t\,\frac{1+{\epsilon}}{1-{\epsilon}}\right],\quad \mbox{whenever}\quad k \leq nt \leq k+1.
\end{equation*}
Indeed, from  \eqref{estim_lil} and \eqref{to_be_obeyed} it follows that, for $k\leq nt\leq k+1$,
\begin{equation*}
t_n\in\left[\frac{h_k^2(\mu)(\bar{g})}{h_n^2(\mu)(\bar{g})},\frac{h_{k+1}^2(\mu)(\bar{g})}{h_n^2(\mu)(\bar{g})}\right]\subset \left[\frac{(1-{\epsilon})\,t\,h_k^2(\mu)(\bar{g})}{(1+{\epsilon})\,h_{k+1}^2(\mu)(\bar{g})},\frac{(1+{\epsilon})\,t\,h_{k+1}^2(\mu)(\bar{g})}{(1-{\epsilon})\,h_k^2(\mu)(\bar{g})}\right],
\end{equation*}
and, according to Lemma \ref{lem:3},
\begin{equation*}
\frac{h_{k+1}^2(\mu)(\bar{g})}{h_k^2(\mu)(\bar{g})}=\frac{h_{k+1}^2(\mu)(\bar{g})}{k+1}\frac{k}{h_k^2(\mu)(\bar{g})}\frac{k+1}{k}
\end{equation*}
converges to $1$ as $n$, and therefore also as $k$, tends to infinity. This finally implies that $t_n\to\infty$, as $n\to \infty$, and thus the desired conclusion follows.

\textbf{Step II. } 
To complete the proof, it suffices to show that 
\begin{equation}\label{eq:0}
\lim_{n\to\infty}\sup_{t\in[0,1]}
\left|\tilde{\eta}_n(\bar{g})(t)-r_n(\bar{g})(t)\right|=0,
\end{equation}
where $(r_n(\bar{g}))_{n\in\mathbb{N}_0}$ is given by \eqref{def:r_n}. 
Indeed, note that \eqref{eq:0}, together with the conclusion of Step I, implies that  $({g}(\phi_n))_{n\in\mathbb{N}_0}$ satisfies the invariance principle for the LIL.

In order to establish \eqref{eq:0}, fix an arbitrary $\bar{\epsilon}>0$ and, for any $k,n\in\mathbb{N}$ with $n>e$ \,(that is $n\geq 4$), define the sets 
\begin{equation*}
A_{k,n}:=
\left\{\frac{\left|M_k(\bar{g})-\sum_{i=0}^{k-1}\bar{g}(\phi_i)\right|}{\sigma(\bar{g})\sqrt{n\ln\ln n}}\geq \bar{\epsilon}/2\right\}
\cup
\left\{\frac{\left|Z_{k+1}(\bar{g})-\bar{g}(\phi_{k})\right|}{\sigma(\bar{g})\sqrt{n\ln\ln n}}\geq\bar{\epsilon}/2\right\}.
\end{equation*}
Further, choose $p\in(2,\infty)$ such that \eqref{eq:abbrev_formula} holds with $\kappa=r$ (specified in the assumptions of the theorem). Using this property, as well as the Markov inequality, we obtain
\begin{align*}
\mathbb{P}_{\mu}\left(\frac{\left|M_k(\bar{g})-\sum_{i=0}^{k-1}\bar{g}(\phi_i)\right|}{\sigma(\bar{g})\sqrt{n\ln\ln n}}\geq \bar{\epsilon}/2\right)
&=\mathbb{P}_{\mu}\left(\frac{\left|\chi(\bar{g})(\phi_k)-\chi(\bar{g})(\phi_0)\right|}{\sigma(\bar{g})\sqrt{n\ln\ln n}}\geq \bar{\epsilon}/2\right)\\
&\leq(2/\bar{\epsilon})^{2+r}p\,\frac{\mathbb{E}_{\mu}|\chi(\bar{g})(\phi_k)|^{2+r}+\mathbb{E}_{\mu}|\chi(\bar{g})(\phi_0)|^{2+r}}{\left(\sigma(\bar{g})\sqrt{n\ln\ln n}\right)^{2+r}}.
\end{align*}
From \eqref{prop_chi} we know that there exist $q\in(0,1)$ and $\hat{c}\in(0,\infty)$ such that, for any $k\in\mathbb{N}$,
\begin{align}\label{eq:oszac}
\begin{aligned}
\mathbb{E}_{\mu}\left|\chi(\bar{g})(\phi_k)\right|^{2+r}
&=\int_X\left|\chi(\bar{g})(u)\right|^{2+r}P^k\mu(du)\\
&\leq p\left|\chi(\bar{g})(\bar{x})\right|^{2+r}
+p\int_X\left|\chi(\bar{g})(u)-\chi(\bar{g})(\bar{x})\right|^{2+r}P^k\mu(du)\\
&\leq p\left|\chi(\bar{g})(\bar{x})\right|^{2+r}
+p^2\left(\frac{2\hat{c}\|\bar{g}\|_{BL}}{1-q}\right)^{2+r}\left(1+\left\langle \varrho^{2+r}(\cdot,\bar{x}),P^k\mu\right\rangle\right).
\end{aligned}
\end{align}
Then, taking into account condition \hyperref[cnd:B1p]{(B1$^{*}$)}, and arguing as in \eqref{estim_}, we obtain 
\begin{align*}
\mathbb{E}_{\mu}\left|\chi(\bar{g})(\phi_k)\right|^{2+r}
&\leq p\left|\chi(\bar{g})(\bar{x})\right|^{2+r}
+p^2\left(\frac{2\hat{c}\|\bar{g}\|_{BL}}{1-q}\right)^{2+r}
\left(1+\left\langle \varrho^{2+r}(\cdot,\bar{x}),\mu\right\rangle^{1/(2+r)}+\frac{b^{*}}{1-a^{*}}\right).
\end{align*}
Hence, for any $k\in\mathbb{N}$ and $n\geq 4$,
\begin{equation}\label{eq:An1}
\mathbb{P}_{\mu}\left(\frac{\left|M_k(\bar{g})-\sum_{i=1}^{k}\bar{g}\left(\phi_i\right)\right|}{\sigma(\bar{g})\sqrt{n\ln\ln n}}\geq \bar{\epsilon}/2\right)\leq \frac{c_1}{\left(\sigma(\bar{g})\sqrt{n\ln\ln n}\right)^{2+r}},
\end{equation}
where $c_1>0$ is some constant independent of $k$ and $n$. Similarly to this, we deduce that 
\begin{equation}\label{eq:An2}
\mathbb{P}_{\mu}\left(\frac{\left|Z_{k+1}(\bar{g})-\bar{g}(\phi_{k})\right|}{\sigma(\bar{g})\sqrt{n\ln\ln n}}\geq\bar{\epsilon}/2\right)\leq \frac{c_2}{\left(\sigma(\bar{g})\sqrt{n\ln\ln n}\right)^{2+r}},
\end{equation}
where $c_2>0$ also does not depend on $k$ and $n$. Clearly, \eqref{eq:An1} and \eqref{eq:An2} imply that $\sum_{n=4}^{\infty}\mathbb{P}_{\mu}(A_{k,n})<\infty$ for every $k\in\mathbb{N}$, and therefore, from the Borel--Cantelli Lemma, 
it follows that 
$\mathbb{P}_{\mu}(\bigcup_{m=4}^{\infty}\bigcap_{n=m}^{\infty} A_{k,n}^{\prime})=1$ for any $k\in\mathbb{N}$. Let
\begin{equation*}
\Omega_0:=\bigcap_{k=1}^{\infty}\bigcup_{m=4}^{\infty}\bigcap_{n=m}^{\infty} A_{k,n}^{\prime}.
\end{equation*}
Obviously, $\mathbb{P}(\Omega_0)=1$. Furthermore, for each $\omega\in\Omega_0$, one can choose $n_0\geq 4>e$ such that 
\begin{equation*}
\sup_{t\in[0,1]}\left|\frac{M_k(\bar{g})-(nt-k)Z_{k+1}(\bar{g})}{\sigma(\bar{g})\sqrt{n\ln\ln n}}-\frac{\sum_{i=0}^{k-1}\bar{g}(\phi_i)+(nt-k)\bar{g}(\phi_{k})}{\sigma(\bar{g})\sqrt{n\ln\ln n}}\right|<\bar{\epsilon}
\end{equation*}
for every $n>n_0$ and any $k\in\{1,\ldots,n-1\}$ satisfying $k<nt\leq k+1$. The proof is now  completed, since $\bar{\epsilon}$ was chosen arbitrarily. 
\end{proof}

\begin{remark}\label{rem}
Analyzing the proof of Theorem \ref{thm:LIL} shows that its assertion remains valid under two more general (and simultaneously, much more abstract) hypotheses, namely:
\begin{itemize}
\item[(i)] condition \hyperref[cnd:B0]{(B0)} and \hyperref[cnd:B1p]{(B1$^*$)} are fulfilled;
\item[(ii)] there exists a Markovian coupling $(\phi^{(1)}_n,\phi^{(2)}_n)_{n\in\mathbb{N}_0}$ of $\Pi$ for which condition \eqref{eq:lemma_thesis} is satisfied.
\end{itemize}
\end{remark}

Let us conclude this section with a brief comparison of the foregoing remark and \cite[Theorem 1]{bms}.
First of all, hypothesis (H2) in \cite{bms}, if formulated for the Fortet--Mourier distance, would be, in fact, equivalent to the existence of a Markovian coupling $(\phi_n^{(1)},\phi_n^{(2)})_{n\in\mathbb{N}_0}$ for which there exist $q\in(0,1)$ and $c\in\mathbb{R}_+$ such that, for any $\mu,\nu\in\mathcal{M}_{1,1}^{\varrho(\cdot, \bar{x})}(X)$,
\begin{align}\label{eq:A}
\left|
\mathbb{E}_{\mu\otimes\nu}
\left(
g\left(\phi_n^{(1)}\right)
-g\left(\phi_n^{(2)}\right)
\right)\right|
\leq \|g\|_{BL}\,cq^nd_{FM}\left(\mu,\nu\right), \quad\text{whenever}\quad g\in Lip_b(X).
\end{align}
On the other hand, conditions \hyperref[cnd:B0]{(B0)}-\hyperref[cnd:B4]{(B5)}, assumed here, imply \eqref{eq:lemma_thesis}, which can be also written in the following form:
\begin{align}\label{B}
\mathbb{E}_{\mu\otimes\nu}
\left|g\left(\phi_n^{(1)}\right)-g\left(\phi_n^{(2)}\right)\right|
\leq \|g\|_{BL}\,c(\mu,\nu)\,q^n\;\;\;\text{for any}\;\;\;g\in Lip_b(X),
\end{align}
where $q\in(0,1)$ and $c(\mu,\nu)\in\mathbb{R}_+$ depends on $\mu$ and $\nu$. Obviously, none of these conditions need not imply the other. However, it is natural to expect that verifying hypothesis \eqref{eq:A}, corresponding to (H2) from \cite{bms}, will usually require establishing the inequality
\begin{align}\label{C}
\mathbb{E}_{\mu\otimes\nu}\left|g\left(\phi_n^{(1)}\right)
-g\left(\phi_n^{(2)}\right)\right|\leq \|g\|_{BL}\,cq^n\,d_{FM}(\mu,\nu),
\end{align}
which is clearly stronger (and more difficult to assure in applications) than condition \eqref{B}, used in this paper.

Moreover, it is not hard to show that hypotheses (H1) and (H3), assumed in \cite{bms}, can be derived from conditions \hyperref[cnd:B1]{(B1)} and \hyperref[cnd:B1p]{(B1$^*$)}, respectively. It should be, however, noted that, in practise, verifying those former would usually come down to checking those latter. 

Let us also stress that, according to  Lemma \ref{cor:g_useful}, hypotheses  \hyperref[cnd:B1]{(B1)}-\hyperref[cnd:B4]{(B5)}, involving the Markov operator~$P$ and a suitable ,,subcoupling'' $Q$ on $X^2\times\mathcal{B}_{X^2}$, do imply condition \eqref{B}, and thus Theorem \ref{thm:LIL} does not demand assuming this property directly. For explicitly defined random dynamical systems, it is quite intuitive how to define $Q$ (see e.g. (6.9) in \cite{dawid}, where the model from Section \ref{sec:ex} is considered, or cf. the proof of \cite[Proposition 3.1]{ks}, concerning a random iterated function system with an arbitrary set of transformations).

Summarizing the above discussion, none of the aforesaid results, that is, \cite[Theorem 1]{bms} and the conclusion formulated in Remark \ref{rem}, need not imply the other, yet the statement given in Remark \ref{rem} may potentially have wider applications than \cite[Theorem 1]{bms} due to the more practical version of condition~(H2).

\section{An Abstract Model for Gene Expression}\label{sec:ex}
In this part of the paper, we intend to apply Theorem \ref{thm:LIL} to a particular random dynamical system, introduced in \cite{dawid}, which provides, among others, a mathematical framework for the analysis of gene expression dynamics (cf. e.g. \cite{dawid,hhs,mtky} for biological interpretation).

Let $(H,\|\cdot\|)$ and $Y$ be a separable Banach space and a~closed subset of this space, respectively. Further, for any $x\in H$ and any $r>0$, let $B(x,r)$ denote an open ball in $H$ centered at $x$ and of radius~$r$. We additionally consider a~topological measure space $(\Theta,\mathcal{B}(\Theta),\Delta)$ with a $\sigma$-finite Borel measure~$\Delta$. With a slight abuse of notation, we will write $d\theta$ instead of $\Delta(d\theta)$ in the rest of the paper. Finally, fix $N\in\mathbb{N}$, and endow the set $I:=\{1,\ldots,N\}$ with the metric $(k,l)\mapsto {d}(k,l)$ given by ${d}(k,l)=1$ for $k\neq l$ and ${d}(k,l)=0$ for $k=l$. 

The subject of our interest will be a discrete-time random dynamical system, defined by the post-jump locations of a piecewise-deterministic stochastic process $(Y(t))_{t\in\mathbb{R}_+}$, evolving on the space $Y$.  The jumps of this process occur at random time points $\tau_n$, \hbox{$n\in\mathbb{N}$}, which coincide with the jump times of a~Poisson process with intensity $\lambda$. In the time intervals $[\tau_{n-1},\tau_{n})$, \hbox{$n\in\mathbb{N}$}, where $\tau_0=0$, the dynamics is deterministically driven by a finite number of semiflows \hbox{$S_i:\mathbb{R}_+\times Y\to Y$}, $i\in I$, which are assumed to be continuous with respect to each variable. The semiflows are switched at the jump times according to a~matrix of continuous probabilities \hbox{$\pi_{ij}:Y\to\left[0,1\right]$}, $i,j\in I$, which satisfy
$\sum_{j\in I}\pi_{ij}(y)=1$ for any $y\in Y$ and $i\in I$. 

The above description can be formalized by putting
\begin{equation}
\label{e:yt}
Y(t):=S_{\xi_n}\left(t-\tau_n,Y\left(\tau_n\right)\right)\;\;\; \mbox{for}\;\;\; t\in[\tau_n,\tau_{n+1}),\;n\in\mathbb{N}_0.\end{equation}
where $\xi_n$ is an $I$-valued random variable  indicating the index of a semiflow chosen directly after the $n$-th jump. The post-jump location $Y(\tau_n)$ is a result of a transformation of the state just before the jump, i.e. $Y(\tau_n-)$, attained by a function randomly selected among all the possible
ones \hbox{$w_{\theta}:Y\to Y$}, \hbox{$\theta\in\Theta$}, and adding a random disturbance $H_n$, which remains within an \hbox{$\varepsilon$-neighbourhood} of zero. Formally, we may therefore write
\begin{equation*}
Y(\tau_n):=w_{\theta_n}(Y(\tau_n-))+H_n \quad\text{for any}\quad n\in\mathbb{N}.
\end{equation*} 

It is required that all the maps $(y,\theta)\mapsto w_{\theta}(y)$ are continuous. Further, we assume that, for some $\varepsilon>0$, all the variables $H_n$, $n\in\mathbb{N}$,  have a common distribution \hbox{$\nu^{\varepsilon}\in\mathcal{M}_1(H)$} supported on $B(0,\varepsilon)\subset H$, and that  
\[w_{\theta}(y)+h\in Y\;\;\; \text{for any} \;\;\; h\in\text{supp}(\nu^{\varepsilon}),\; \theta \in \Theta,\;y\in Y.\]
Moreover, the probability of choosing $w_{\theta}$ (at the jump time $\tau_n$) is determined by a density \hbox{$\theta \mapsto p(y,\theta)$}, whenever $Y(\tau_n-)=y$, where \hbox{$p: Y\times \Theta \to \left[0,\infty\right)$} is a given continuous function satisfying 
$\int_{\Theta} p(y,\theta)\,d\theta=1$ for any $y\in Y$.

Now, consider the set $X:=Y\times I$, endowed with the metric of the form
\begin{equation}\label{dfm}\varrho_{\tilde{c}}\left((y_1,i),(y_2,j)\right)=\|y_1-y_2\|+\tilde{c}\,{d}(i,j)\quad\mbox{for}\quad (y_1,i),(y_2,j)\in X,\end{equation}
where $\tilde{c}$ is a positive constant.
The main goal of this section is to establish the functional LIL for the sequence of random variables $(Y_n,\xi_n)_{n\in\mathbb{N}_0}$ with values in $X$, where $Y_n:=Y(\tau_n)$ and $\xi_n$ is the $I$-valued random variable appearing in \eqref{e:yt}. The joint distribution of $(Y_0,\xi_0)$ will be denoted by $\mu\in\mathcal{M}_1(X)$. The sequence $(Y_n,\xi_n)_{n\in\mathbb{N}_0}$ can be defined on an appropriate probability space, say $(\Omega, \mathcal{F}, \mathbb{P}_{\mu})$, in such a way that, for every $n\in\mathbb{N}_0$,
\begin{equation}\label{def:Y_n}
Y_{n+1}=w_{\theta_{n+1}}(S_{\xi_{n}}(\Delta \tau_{n+1},Y_{n}))+H_{n+1}, \quad\text{where}\quad \Delta\tau_{n+1}:=\tau_{n+1}-\tau_n,
\end{equation}
whilst $(\xi_n)_{n\in\mathbb{N}_0}$, taking values in $I$, and the remaining, auxiliary sequences of random variables $(\tau_n)_{n\in\mathbb{N}_0}$, $(\theta_n)_{n\in\mathbb{N}}$ and $(H_n)_{n\in\mathbb{N}}$ with values in $\mathbb{R}_+$, $\Theta$ and $H$, respectively, are specified by the following conditions:
\begin{itemize}

\item[(i)]  $\tau_0=0$, $\tau_n\uparrow \infty$ almost surely, as $n\to \infty$, and the increments $\Delta \tau_n$, $n\in\mathbb{N}$, are mutually independent and have the common exponential distribution with intensity $\lambda>0$;
\item[(ii)] $H_n$, $n\in\mathbb{N}$, are identically distributed with~$\nu^{\varepsilon}$;
\item[(iii)] $\theta_n$ and $\xi_n$, $n\in\mathbb{N}$, are defined inductively, so that
\begin{gather*}
\mathbb{P}_{\mu}(\theta_{n+1}\in D\;|\;S_{\xi_{n}}(\Delta \tau_{n+1},Y_{n})=y;\,W_n)
=\int_{D} p(y,\theta)\,d\theta, \\
\mathbb{P}_{\mu}(\xi_{n+1}=j\;|\;Y_{n+1}=y,\, \xi_{n}=i;\,W_n)=\pi_{ij}(y) \;\;\;\text{for any}\;\;\;D\in\mathcal{B}(\Theta),\;y\in Y,\;i, j\in I,\;n\in\mathbb{N}_0,
\end{gather*}
where 
\begin{equation*}
W_0=(Y_0,\;\xi_0)\;\;\;\text{and}\;\;\;W_n=(W_0,\;H_1,\ldots,H_n,\;\tau_1,\ldots,\tau_n,\;\theta_1,\ldots,\theta_n,\;\xi_1,\ldots,\xi_n)\;\;\;\text{for}\;\;\;n\in\mathbb{N}.
\end{equation*}
\end{itemize}
We also demand that, for any $n\in\mathbb{N}_0$, the variables $\Delta\tau_{n+1}$, $H_{n+1}$, $\theta_{n+1}$ and $\xi_{n+1}$ are ({mutually}) conditionally independent given $W_n$, and that $\Delta\tau_{n+1}$ and $H_{n+1}$ are independent of $W_n$.

An easy computation shows that  $(Y_n,\xi_n)_{n\in\mathbb{N}_0}$ is a time-homogeneous Markov chain with transition law $\Pi:X\times\mathcal{B}(X)\to\left[0,1\right]$ given by
\begin{align}\label{def:Pi_epsilon}
\Pi((y,i),A)=&\int_0^{\infty}\lambda e^{-\lambda t}\int_{\Theta}p(S_i(t,y),\theta) 
\nonumber\\
&\times \int_{\supp\,\nu^{\varepsilon}}\left(\sum_{j\in I}\mathbbm{1}_A(w_{\theta}(S_i(t,y))+h,j)\,\pi_{ij}(w_{\theta}(S_i(t,y))+h)\right) \nu^{\varepsilon}(dh)\,d\theta\,dt\quad 
\end{align}
for any $(y,i)\in X$ and any $A\in\mathcal{B}_X$. 

Let us now detail the conditions that have been employed in \cite{dawid} (cf. also \cite{asia}), in order to establish the exponential ergodicity for the chain $(Y_n,\xi_n)_{n\in\mathbb{N}_0}$ in the Fortet--Mourier distance. Namely, it is assumed that there exist $\bar{y}\in Y$, a function $\mathcal{L}:Y\to \mathbb{R}_+$ that is bounded on bounded sets, and constants $\alpha\in\mathbb{R}$, $L,\,L_w,\,L_{\pi},\,L_p,\,c_{\pi},\,c_p>0$ such that
\begin{equation}\label{eq:balance1} LL_w+\alpha/\lambda<1, \end{equation}
and, for any $i,i_1,i_2\in I$, $y_1,y_2\in Y$, $t\geq 0$, the following conditions hold:
\begin{gather}
\tag{A1}\label{cnd:a1}
\sup_{y\in Y}\int_0^{\infty}e^{-\lambda t}\int_{\Theta}\left\| w_{\theta}(S_i(t,\bar{y}))-\bar{y}\right\|p(S_i(t,y),\theta)\,d\theta\,dt<\infty,\\
\tag{A2}\label{cnd:a2}
\left\|S_{i_1}(t,y_1)-S_{i_2}(t,y_2)\right\|\leq Le^{\alpha t}\left\|y_1-y_2\right\|+t\mathcal{L}(y_2)\,d(i_1,i_2),\\
\tag{A3}\label{cnd:a3}
\int_{\Theta}\left\|w_{\theta}(y_1)-w_{\theta}(y_2)\right\| p(y_1,\theta)\,d\theta\leq L_w\left\|y_1-y_2\right\|,\\
\tag{A4}\label{cnd:a4}
\sum_{j\in I} |\pi_{ij}(y_1)-\pi_{ij}(y_2)|\leq L_{\pi} \left\|y_1-y_2\right\|,\;\;\;\int_{\Theta} |p(y_1,\theta)-p(y_2,\theta)|\,d\theta\leq L_{p} \left\|y_1-y_2\right\|,\\
\tag{A5}\label{cnd:a5}
\sum_{j\in I} \pi_{i_1j}(y_1)\wedge \pi_{i_2j}(y_2)\geq c_{\pi},\;\;\;\int_{\Theta(y_1,y_2)}p(y_1,\theta)\wedge p(y_2,\theta)\,d\theta\geq c_p,
\end{gather}
where $\Theta(y_1,y_2):=\{\theta\in\Theta:\, \|w_{\theta}(y_1)-w_{\theta}(y_2)\|\leq L_w \|y_1-y_2\|\}$. 

From the proof of \cite[Theorem 4.1]{dawid} it follows that, if conditions \hyperref[cnd:a1]{(A1)}-\hyperref[cnd:a5]{(A5)} hold with $\alpha$, $L$ and $L_w$ satisfying~\eqref{eq:balance1}, and the constant $\tilde{c}$, appearing in \eqref{dfm}, is sufficiently large (according to the constants in the hypotheses above; cf. \cite{dawid}), then the assumptions of \hbox{Theorem~\ref{thm:spectral_gap}}, i.e. \hyperref[cnd:B0]{(B0)}-\hyperref[cnd:B4]{(B5)}, are fulfilled for the Markov operator $P$ corresponding to $\Pi$, defined by \eqref{def:Pi_epsilon}, and some substochastic kernel $Q$ on $X^2\times \mathcal{B}_{X^2}$. The latter, for any  $(y_1,i_1), (y_2,i_2)\in X$ and $A\in\mathcal{B}_{X^2}$, is given by
\begin{align}\label{def:Q_ex}
\begin{aligned}
Q\left(\left(y_1,i_1\right),\left(y_2,i_2\right),A\right)
=&\int_0^{\infty}\lambda e^{-\lambda t}
\int_{\Theta}p\left(S_{i_1}\left(t,y_1\right),\theta\right)\wedge 
p\left(S_{i_2}\left(t,y_2\right),\theta\right)
\\
&\times\int_{\supp\,\nu^{\varepsilon}}\Bigg(\sum_{j\in I}\mathbbm{1}_A
\left(
\left(w_{\theta}\left(S_{i_1}(t,y_1)\right)+h,j\right),\left(w_{\theta}\left(S_{i_2}(t,y_2)\right)+h,j\right)\right)
\\
&\times 
\pi_{i_1j}\left(w_{\theta}\left(S_{i_1}\left(t,y_1\right)\right)+h\right)\wedge
\pi_{i_2j}\left(w_{\theta}\left(S_{i_2}\left(t,y_2\right)\right)+h\right)\Bigg)
\,\nu^{\varepsilon}(dh)\,d\theta\,dt.
\end{aligned}
\end{align} 
Consequently,  $P$ is then exponentially ergodic in $d_{FM}$ induced by the metric $\varrho_{\tilde{c}}$, defined by \eqref{dfm}.

It should be pointed out here that, in the proof of \cite[Theorem 4.1]{dawid}, condition \hyperref[cnd:B1]{(B1)} has been verified for the \hbox{Lyapunov} function of the form
\begin{equation}
\label{def:V_erg}
\bar{V}(x):= \|y-\bar{y}\|\quad\text{for}\quad x=(y,i)\in X
\end{equation}
with $\bar{y}\in Y$ determined by \hyperref[cnd:a1]{(A1)}. However, let us observe that \hyperref[cnd:B1]{(B1)} must be also fulfilled for the Lyapunov function $V_{\tilde{c}}$, considered in this case, which is defined by
\begin{align}
\label{def:V_lil}
V_{\tilde{c}}(x):=\varrho_{\tilde{c}}(x,\bar{x})\quad\mbox{for every}\quad x\in X,
\end{align}
where $\bar{x}:=\left(\bar{y},\bar{i}\right)$, and $\bar{i}$ is an arbitrarily fixed element of $I$, which is in accordance with the assumptions imposed in Section \ref{sec:lil}, wherein Theorem \ref{thm:LIL} is stated.

Willing to verify the Strassen invariance principle for the LIL, we need to strengthen conditions \eqref{eq:balance1}, \hyperref[cnd:a1]{(A1)} and \hyperref[cnd:a3]{(A3)}. Namely, we require that there exist $r\in(0,2)$ and $L_w^{*}>0$, satisfying
\begin{equation}
\label{LIL_condition}
L^{2+r}L_w^{*}+(2+r)\alpha\lambda^{-1}<1
\end{equation}
such that, for some $\bar{y}\in Y$, the following statements hold:
\begin{gather}
\tag{A1$^*$}\label{cnd_strong:A1}
\sup_{y\in Y}\int_0^{\infty}e^{-\lambda t}\int_{\Theta}\|{ w_{\theta}(S_i(t,\bar{y}))-\bar{y}}\|^{2+r}p(S_i(t,y),\theta)\,d\theta\,dt<\infty \quad\text{for any}\quad i\in I.\\
\tag{A3$^*$}\label{cnd_strong:A3}
\int_{\Theta} \|{w_{\theta}(y_1)-w_{\theta}(y_2)}\|^{2+r}p(y_1,\theta)\,d\theta\leq L_w^{*}\|{y_1-y_2}\|^{2+r}\quad\text{for any}\quad y_1,y_2\in Y.
\end{gather}

\begin{remark}\label{rem:implication}
Due to the H\"older inequality, conditions \hyperref[cnd_strong:A1]{(A1$^*$)} and \hyperref[cnd_strong:A3]{(A3$^*$)} imply \hyperref[cnd:a1]{(A1)} and \hyperref[cnd:a3]{(A3)}, respectively,  and the latter hold with $L_w:=(L_w^*)^{1/(2+r)}$. 

Furthermore, let us observe that inequality \eqref{LIL_condition} implies \eqref{eq:balance1}. To see this, suppose, conversly to \eqref{eq:balance1}, that \hbox{$LL_w+\alpha\lambda^{-1}\geq 1$}. Then, noting that \hbox{$\alpha\lambda^{-1} < (2+r)^{-1}<1$}, we obtain $(LL_w)^{2+r}\geq(1-\alpha\lambda^{-1})^{2+r}$, which due to the Bernoulli inquality, leads to the contradiction with \eqref{LIL_condition}.
\end{remark}

\begin{theorem}\label{thm_LIL}
Let $(Y_n,\xi_n)_{n\in\mathbb{N}_0}$ be the Markov chain with transition law $\Pi$ given by \eqref{def:Pi_epsilon} and initial distribution $\mu\in\mathcal{M}_{1}(X)$. Further, assume that conditions \hyperref[cnd:a1]{(A1)}-\hyperref[cnd:a5]{(A5)} with \hyperref[cnd:a1]{(A1)} and \hyperref[cnd:a3]{(A3)} strengthened to \hyperref[cnd_strong:A1]{(A1$^*$)} and \hyperref[cnd_strong:A3]{(A3$^*$)}, respectively, hold with constants $\alpha\in\mathbb{R}$, $L, L_w^*>0$  and $r\in (0,2)$ satisfying \eqref{LIL_condition}. Then, for every $g\in Lip_b(X)$ with $\sigma^2(\bar{g})>0$, the chain $(g(Y_n,\xi_n))_{n\in\mathbb{N}_0}$ obeys the invariance principle for the LIL, provided that $\mu\in\mathcal{M}_{1,2+r}^{\bar{V}}(X)$ for the Lyapunov function $\bar{V}$, given by \eqref{def:V_erg}.
\end{theorem}
\begin{proof}
We intend to apply our criterion on the invariance principle for the LIL, stated as Theorem \ref{thm:LIL}, for the Markov operator $P$ induced by $\Pi$ and the substochastic kernel $Q$ given by \eqref{def:Q_ex}. To this end, let us first observe that \hbox{$\mathcal{M}_{1,2+r}^{\bar{V}}(X)=\mathcal{M}_{1,2+r}^{V_{\tilde{c}}}(X)$}, where $V_{\tilde{c}}$ is defined by \eqref{def:V_erg}. This yields that \hbox{$\mu\in \mathcal{M}_{1,2+r}^{V_{\tilde{c}}}(X)$}. Moreover, as mentioned earlier, conditions \hyperref[cnd:B0]{(B0)}-\hyperref[cnd:B4]{(B5)} can be derived from \hyperref[cnd:a1]{(A1)}-\hyperref[cnd:a5]{(A5)}, fulfilled with $L$, $L_w$ and $\alpha$ satisfying \eqref{eq:balance1} (as it was shown in the proof of \cite[Theorem 4.1]{dawid}), and so, according to Remark~\ref{rem:implication}, they can be also derived from the assumptions of this theorem.

In the light of the above, the proof of Theorem \ref{thm_LIL} reduces to showing \hyperref[cnd:B1p]{(B1$^{*}$)}. In order to do this, first of all, note that
\begin{align}\label{eq}
\begin{aligned}
\big\langle &\varrho_c^{2+r}(\cdot,\bar{x}),P\mu \big\rangle 
=\int_X\int_X 
\varrho_c^{2+r}\left((z,l),(\bar{y},\bar{i})\right)\,\Pi\left((y,i),dz\times dl\right)\mu(dy\times di)\\
&=\int_X\int_0^{\infty}\lambda e^{-\lambda t}\int_{\Theta}p(S_i(t,y),\theta)\int_{\supp\,\nu^{\varepsilon}}
\Bigg(\sum_{j\in I}
\Big(\left\|w_{\theta}(S_i(t,y))+h-\bar{y}\right\|+cd(j,\bar{i})\Big)^{2+r}\,\\
&\quad\times\pi_{ij}\left(w_{\theta}(S_i(t,y))+h\right)\Bigg)
\nu^{\varepsilon}(dh)\,d\theta \,dt\,\mu(dy\times di).
\end{aligned}
\end{align}
Now, introduce $Z:=X\times[0,\infty)\times\Theta\times \supp \nu^{\varepsilon}\times I$ (where $X=Y\times I$), and define $\nu\in \mathcal{M}_1(Z)$ as follows:

\begin{align*}
\nu(A)
&:=\int_X\int_0^{\infty}\lambda e^{-\lambda t}\int_{\Theta}p(S_i(t,y),\theta)\int_{\supp\,\nu^{\varepsilon}}\left(\sum_{j\in I}\mathbbm{l}_A(y,i,t,\theta,h,j)\,\pi_{ij}\left(w_{\theta}(S_i(t,y))+h\right)\right)\\
&\quad\times\nu^{\varepsilon}(dh)\,  d\theta\,dt\,\mu(dy\times di)\qquad\mbox{for any}\qquad A\in\mathcal{B}_Z.
\end{align*}
Let us further consider $\varphi_0:Z\to\mathbb{R}$ given by
\begin{align*}
\varphi_0(y,i,t,\theta,h,j):=\left\|w_{\theta}(S_i(t,y))+h-\bar{y}\right\|+\tilde{c}d(j,\bar{i})
\end{align*}
for any $(y,i)\in X$, $t\in\mathbb{R}_+$, $\theta\in\Theta$, $h\in \supp\nu^{\varepsilon}$ and $j\in I$. Observe that $\varphi_0$ is a non-negative Borel measurable function, and that
\begin{align*}
\varphi_0(y,i,t,\theta,h,j)&\leq \left\|w_{\theta}(S_i(t,y))-w_{\theta}(S_i(t,\bar{y}))\right\|
+\left\|w_{\theta}(S_i(t,\bar{y}))-\bar{y}\right\|+\|h\|+\tilde{c}d(j,\bar{i}).
\end{align*}
Consequently, using the Minkowski inequality, we obtain
\begin{align}\label{eq:LIL_model}
\begin{aligned}
&\left\langle \varrho_c^{2+r}(\cdot,\bar{x}),P\mu \right\rangle ^{1/(2+r)}
=\left(\int_Z\varphi_0^{2+r}(y,i,t,\theta,h,j)\,\nu(dy\times di\times dt\times d\theta\times dh\times dj)\right)^{1/(2+r)}
\\
&\qquad\leq \left(\int_{Z}
\left\|w_{\theta}(S_i(t,y))-w_{\theta}(S_i(t,\bar{y}))\right\|^{2+r}\nu(dy\times di\times dt\times d\theta\times dh\times dj)\right)^{1/(2+r)}\\
&\qquad\quad+\left(\int_{Z}\left\|w_{\theta}(S_i(t,\bar{y}))-\bar{y}\right\|^{2+r}\nu(dy\times di\times dt\times d\theta\times dh\times dj)\right)^{1/(2+r)}+\varepsilon+\tilde{c},
\end{aligned}
\end{align}
where the second component on the right-hand side of the above inequality is finite due to \hyperref[cnd_strong:A1]{(A1$^*$)}. 
\hbox{According} to assumptions \hyperref[cnd_strong:A3]{(A3$^*$)} and \hyperref[cnd:a2]{(A2)}, we further have
\begin{align}
\label{eq:LIL_model_estim}
\begin{aligned}
&\int_{Z}
\left\|w_{\theta}(S_i(t,y))-w_{\theta}(S_i(t,\bar{y}))\right\|^{2+r}\nu(dy\times di\times dt\times d\theta\times dh\times dj)\\
&\qquad\leq
\int_X\int_0^{\infty}\lambda e^{-\lambda t} L_w^{*}\|S_i(t,y)-S_i(t,\bar{y})\|^{2+r}\,dt\,\mu(dy\times di)\\
&\qquad\leq 
\int_X\int_0^{\infty}\lambda e^{-\lambda t} L_w^{*}L^{2+r}e^{(2+r)\alpha t}\|y-\bar{y}\|^{2+r} \,dt\,\mu(dy\times di)\\
&\qquad\leq
\lambda L_w^{*}L^{2+r}\left(\int_0^{\infty}e^{-(\lambda-(2+r)\alpha)t}\,dt\right)\left(\int_X
\|y-\bar{y}\|^{2+r}
\,\mu(dy\times di)\right)\\
&\qquad\leq
\frac{\lambda L_w^{*}L^{2+r}}{\lambda-(2+r)\alpha}\left\langle V_{\tilde{c}}^{2+r},\mu\right\rangle,
\end{aligned}
\end{align}
where the last inequality follows from the fact that $(2+r)\alpha<\lambda$, which is provided by~(\ref{LIL_condition}). 
Hence, referring to (\ref{eq:LIL_model}) and (\ref{eq:LIL_model_estim}), we obtain condition \hyperref[cnd:B1p]{(B1$^{*}$)} with 
\begin{gather*}
a^{*}:=\frac{\lambda L_w^{*}L^{2+r}}{\lambda-(2+r)\alpha},\\
b^{*}:=\sup_{y\in Y}\left|\int_0^{\infty}e^{-\lambda t}\int_{\Theta}\|{ w_{\theta}(S_i(t,\bar{y}))-\bar{y}}\|^{2+r}p(S_i(t,y),\theta)\,d\theta\,dt\right|^{1/(2+r)}+\varepsilon_*+\tilde{c}.
\end{gather*}
Moreover, due to condition (\ref{LIL_condition}),  we see that $a^{*}\in(0,1)$, which completes the proof.
\end{proof}

An important special case of the above-discussed Markov chain, obtained by putting $I:=\{1\}$ and $S_1(t,y):=y$, is a~random iterated function system with an additive disturbance (see \cite{hhsw2}), which occurs e.g. in a stochastic model of single-gene autoregulation (described in \cite{hhs}). In this setting, $(Y_n,\xi_n)_{n\in\mathbb{N}_0}$, evolving on $X=Y\times \{1\}$, can be identified with the $Y$-valued chain $(Y_n)_{n\in\mathbb{N}_0}$, which takes the form:
\begin{equation*}
Y_{n+1}= w_{\theta_{n+1}}(Y_n)+H_{n+1}.
\end{equation*}
The one-step transition law is then given by
\begin{align}\label{def_Pi}
\Pi(y,A)=\int_{\Theta} p(y,\theta)\int_{\supp\,\nu^{\varepsilon}}\mathbbm{1}_A(w_{\theta}(y)+h)\,\nu^{\varepsilon}(dh)\,d\theta\;\;\;\text{for any}\;\;\; y\in Y,\,A\in \mathcal{B}_Y.
\end{align}

In the case where no disturbance occurs, i.e.  $H_n=0$ for all $n\in\mathbb{N}$, the system reduces to a standard random iterated function system, which can serve, for instance, as a model of cell cycle (see \cite{lm,hw}). A bit more general version of such a system is also investigated in \cite{ks}.

In this particular situation, conditions \eqref{eq:balance1} and \hyperref[cnd:a1]{(A1)}-\hyperref[cnd:a5]{(A5)}, guaranteeing the exponential ergodicity in~$d_{FM}$, can be simplified to the following requirements: there exist $\bar{y}\in Y$, $L_w\in (0,1)$, $L_p>0$ and $c_p>0$ such that, for any $y,y_1,y_2\in Y$, we have
\begin{gather}
\tag{A1'}\label{cnd:i}
\sup_{y\in Y}\int_{\Theta}\left\|w_{\theta}(\bar{y})-\bar{y}\right\|p(y,\theta)\,d\theta<\infty,
\\
\tag{A3'}\label{cnd:ii}
\int_{\Theta}p(y,\theta) \left\|w_{\theta}(y_1)-w_{\theta}(y_2)\right\|\,d\theta\leq L_w \|y_1-y_2\|,
\\
\tag{A4'}\label{cnd:iii}
\int_{\Theta}|p(y_1,\theta)-p(y_2,\theta)|\,d\theta\leq L_p\|y_1-y_2\|,
\\
\tag{A5'}\label{cnd:iv}
\int_{\Theta(y_1,y_2)}p(y_1,\theta)\wedge p(y_2,\theta)\,d\theta\geq c_p,
\end{gather}
where $\Theta(y_1,y_2)=\{\theta\in\Theta:\,\|w_{\theta}(y_1)-w_{\theta}(y_2)\|\leq L_w\|y_1-y_2\|\}$.
Note that \eqref{eq:balance1} is then trivially satisfied, since $L_w<1$, and we can take $\alpha=0$ and $L=1$.

Clearly, in order to apply Theorem \ref{thm_LIL}, i.e. our criterion on the functional LIL, we need to assume the appropriately simplified versions of \hyperref[cnd_strong:A1]{(A1$^*$)} and \hyperref[cnd_strong:A3]{(A3$^*$)}, rather that \hyperref[cnd:a1]{(A1)} and \hyperref[cnd:a3]{(A3)}, that is, the existence of $L_w^*\in(0,1)$ and $r\in(0,2)$ such that
\begin{gather}
\tag{A1$^*$'}\label{cnd_strong:i}
\sup_{y\in Y}\int_{\Theta}\left\|w_{\theta}(\bar{y})-\bar{y}\right\|^{2+r}p(y,\theta)\,d\theta<\infty,
\\
\tag{A3$^*$'}\label{cnd_strong:ii}
\int_{\Theta}p(y,\theta) \left\|w_{\theta}(y_1)-w_{\theta}(y_2)\right\|^{2+r}\,d\theta\leq L_w^* \|y_1-y_2\|,\;\;\;y_1,y_2\in Y.
\end{gather}
Obviously, inequality \eqref{LIL_condition} then also holds trivially.

While the proof of the earlier-mentioned \cite[Theorem 4.1]{dawid}, guaranteeing that conditions \hyperref[cnd:B0]{(B0)}-\hyperref[cnd:B4]{(B5)} hold for the general model, is rather long and technical, in this particular case, these conditions can be derived directly in a relatively simple way (cf. \cite[Proposition 3.1]{ks}). More specifically, from the continuity of $w_{\theta}$ and $p(\cdot,\theta)$ it follows that the Markov operator $P$ (corresponding to $\Pi$, given by \eqref{def:Pi_epsilon}) enjoys the Feller property, i.e. \hyperref[cnd:B0]{(B0)} holds. If we now consider the appropriately simplified form of the kernel $Q$, defined on $Y^2\times\mathcal{B}_{Y^2}$ by
\begin{align*}
Q(y_1,y_2,A):=\int_{\Theta}p(y_1,\theta)\wedge p(y_2,\theta)\int_{\supp\,\nu^{\varepsilon}}\mathbbm{1}_A(w_{\theta}(y_1)+h,w_{\theta}(y_2)+h)\,\nu^{\varepsilon}(dh)\,d\theta
\end{align*}
for any $(y_1,y_2)\in Y^2$, $A\in\mathcal{B}_{Y^2}$, 
and we take $F=Y^2$, then hypotheses \hyperref[cnd:B1]{(B1)}, \hyperref[cnd:B3]{(B3)}, \hyperref[cnd:B4]{(B4)} and \hyperref[cnd:B5]{(B5)} can be deduced almost immediately from \hyperref[cnd:i]{(A1')} \& \hyperref[cnd:ii]{(A3')},\, \hyperref[cnd:ii]{(A3')}, \hyperref[cnd:iv]{(A5')} and \hyperref[cnd:iii]{(A4')}, respectively, while \hyperref[cnd:B5]{(B2)} is trivially satisfied (as the domain of contractivity is $Y^2$). Finally, to prove the LIL, it suffices to note that \hyperref[cnd:B1p]{(B1$^*$)} follows from \hyperref[cnd_strong:i]{(A1$^*$')} and \hyperref[cnd_strong:ii]{(A3$^*$')}.

\section*{Appendix}\label{append}
Within the appendix, we present the proofs of lemmas from Section \ref{subsec}.

\begin{proof}[Proof of Lemma \ref{lem:constant}]
Fix an arbitrary $m\in\mathbb{N}\cup\{\infty\}$ and $c\geq 0$. We shall give the proof for $f_{m,c}^{\inf}$. 
The reasoning for $f_{m,c}^{\sup}$ is analogous. 
Note that, for every $n\in\mathbb{N}$ and every $k\in\mathbb{N}$, we have
\begin{align*}
\frac{1}{n}\sum_{l=k+1}^n\left(z_l^2\wedge m\right)=\frac{1}{n}\sum_{l=1}^n\left(z_l^2\wedge m\right)-\frac{1}{n}\sum_{l=1}^k\left(z_l^2\wedge m\right),
\end{align*}
and therefore
\begin{align}\label{eq:m_shift}
\liminf_{n\to\infty}\frac{1}{n}\sum_{l=k+1}^n\left(z_l^2\wedge m\right)
=\liminf_{n\to\infty}\frac{1}{n}\sum_{l=1}^n\left(z_l^2\wedge m\right).
\end{align}
Hence, for an arbitrarily fixed $l_0\in\mathbb{N}\backslash\{1\}$, we obtain
\begin{align*}
f_{m,c}^{\inf}(x)&=\mathbb{E}_x\left(\left|\liminf_{n\to\infty}\left(\frac{1}{n}\sum_{l=l_0}^n \left(z_l^2\wedge m\right) \right)-c\right|\wedge 1\right)\\
&=\lim_{n\to\infty}\mathbb{E}_x\left(\left|\inf_{k\geq n}\left(\frac{1}{k}\sum_{l=l_0}^k \left(z_l^2\wedge m\right) \right)-c\right|\wedge 1\right)\\
&=\lim_{n\to\infty}\lim_{N\to\infty}\mathbb{E}_x\left(\left|\min_{k\in\{n,n+1,\ldots,n+N\}}\left(\frac{1}{k}\sum_{l=l_0}^k \left(z_l^2\wedge m\right) \right)-c\right|\wedge 1\right).
\end{align*}
Consequently, defining $H_{n,N}:X\to\mathbb{R}$ for $n,N\in\mathbb{N}$ and $n\geq l_0$, by the formula
\begin{equation*}
H_{n,N}(x)=\mathbb{E}_x\left(\left|\min_{k\in\{n,n+1,\ldots,n+N\}}\left(\frac{1}{k}\sum_{l=l_0}^k \left(z_l^2\wedge m\right) \right)-c\right|\wedge 1\right),\end{equation*}
we get
\begin{equation}
\label{form_i}
f_{m,c}^{\inf}(x)=\lim_{n\to\infty}\lim_{N\to\infty} H_{n,N}(x)\;\;\;\mbox{for every}\;\;\;x\in X.
\end{equation}
Let us now observe that, for any $\alpha_i,\lambda\in\mathbb{R}$, where $i\in I$ and $I$ is a nonempty, finite set, we have
\begin{equation*}
|\min_{i\in I} \alpha_i-\lambda|\wedge 1 = |\min_{i\in I} (\alpha_i \wedge (1+\lambda))-\lambda|\wedge 1.\end{equation*} 
This in turn implies that, for $x\in X$,
\begin{align*}
H_{n,N}(x)&=\mathbb{E}_x\left(\left|\min_{k\in\{n,n+1,\ldots,n+N\}}\left(\left(\frac{1}{k}\sum_{l=l_0}^k \left(z_l^2\wedge m\right) \right)\wedge (1+c) \right)-c\right|\wedge 1\right)\\
&=\mathbb{E}_x\left(\left|\min_{k\in\{n,n+1,\ldots,n+N\}}\frac{1}{k}\left(\left(\sum_{l=l_0}^k \left(z_l^2\wedge m\right) \right)\wedge k(1+c) \right)-c\right|\wedge 1\right).
\end{align*}
For every pair $(n,N)$ such that $n,N\in \mathbb{N}$ and $n\geq l_0$, let us consider a~random variable $\Psi_{n,N}:\Omega\to X$ given by
\begin{equation*}\Psi_{n,N}=\min_{k\in\{n,n+1,\ldots,n+N\}}\frac{1}{k}\left(\left(\sum_{l=l_0}^k \left(z_l^2\wedge m\right) \right)\wedge k(1+c) \right)-c.\end{equation*}
Then $H_{n,N}(x)=\mathbb{E}_x(|\Psi_{n,N}|\wedge 1)$, $x\in X$, and hence
\begin{align}\label{eq:H-H}
\begin{aligned}
|H_{n,N}(x)-H_{n,N}(y)|
&\leq \mathbb{E}_{x,y}\left|(|\Psi_{n,N}^{(1)}|\wedge 1)-(|\Psi_{n,N}^{(2)}|\wedge 1)\right|\\
&\leq \mathbb{E}_{x,y}\left||\Psi_{n,N}^{(1)}|-|\Psi_{n,N}^{(2)}|\right|\leq \mathbb{E}_{x,y}\left|\Psi_{n,N}^{(1)}-\Psi_{n,N}^{(2)}\right|,
\end{aligned}
\end{align}
where the second inequality is implied by the following general property: 
\begin{align}\label{general_fact}
|\alpha\wedge c-\lambda\wedge c|\leq|\alpha-\lambda|\;\;\;\mbox{for any}\;\;\;\alpha,\lambda,c\in\mathbb{R}_+.
\end{align}
Now, since for any $c\in\mathbb{R}$ and all $\alpha_i,\lambda_i\in\mathbb{R}_+$, $i\in I$, where $I$ is a nonempty finite set,
\begin{align*}
|\min_{i\in I} \alpha_i - \min_{i\in I} \lambda_i|\leq \max_{i\in I} |\alpha_i - \lambda_i|
\end{align*}
and
\begin{align*}
\left|\left(\sum_{i\in I}\alpha_i\right)\wedge c-\left(\sum_{i\in I}\lambda_i\right)\wedge c\right|\leq\sum_{i\in I}|\alpha_i\wedge c-\lambda_i\wedge c|,
\end{align*}
we see that
\begin{align*}
&\left|\Psi_{n,N}^{(1)}-\Psi_{n,N}^{(2)}\right|\\
&\hspace{0.4cm}\leq \max_{k\in\{n,\ldots,n+N\}} \frac{1}{k}\Bigg|\left(\sum_{l=l_0}^k \left(z_l^{(1)}\right)^2 \wedge m \right)\wedge k(1+c) 
-\left(\sum_{l=l_0}^k \left(z_l^{(2)}\right)^2\wedge m \right)\wedge k(1+c)  \Bigg|\\
&\hspace{0.4cm}\leq \max_{k\in\{n,\ldots,n+N\}} \frac{1}{k}\sum_{l=l_0}^k\left|\left(z_l^{(1)}\right)^2\wedge m\wedge k(1+c)-\left(z_l^{(2)}\right)^2\wedge m\wedge k(1+c)\right|\\
&\hspace{0.4cm}\leq \max_{k\in\{n,\ldots,n+N\}} \frac{1}{k}\sum_{l=l_0}^k\left|\left(z_l^{(1)}\right)^2\wedge k(1+c)-\left(z_l^{(2)}\right)^2\wedge k(1+c)\right|,
\end{align*}
where the last inequality follows from \eqref{general_fact}.  
Further, applying the inequality 
\begin{equation*}\left|\alpha_1^2\wedge \lambda-\alpha_2^2\wedge \lambda\right|\leq 2\sqrt{\lambda}|\alpha_1-\alpha_2|,\;\;\;\mbox{where}\;\;\;\alpha_1,\alpha_2\in\mathbb{R}\;\;\;\mbox{and}\;\;\;\lambda\in\mathbb{R}_+,\end{equation*} 
we obtain 
\begin{align*}
\left|\Psi_{n,N}^{(1)}-\Psi_{n,N}^{(2)}\right|
&\leq \max_{k\in\{n,n+1,\ldots,n+N\}} \frac{1}{k}2k(1+c)\sum_{l=l_0}^k \left|z_l^{(1)}-z_l^{(2)} \right|
= 2(1+c)\sum_{l=l_0}^{n+N}\left|z_l^{(1)}-z_l^{(2)}\right|.
\end{align*}
From the above estimation and \eqref{eq:H-H} it follows that, for $x,y\in X$, $n,N\in\mathbb{N}$, $n\geq l_0$,
\begin{align}
\label{ggg}
|H_{n,N}(x)-H_{n,N}(y)|\leq 2(1+c)\sum_{l=l_0}^{n+N}\mathbb{E}_{x,y}\left|z_l^{(1)}-z_l^{(2)}\right|.
\end{align}
Hence, according to \eqref{form_i}, we have
\begin{equation*}|f_{m,c}^{\inf}(x)-f_{m,c}^{\inf}(y)|\leq 2(1+c)\sum_{l=l_0}^{\infty}\mathbb{E}_{x,y}\left|z_l^{(1)}-z_l^{(2)}\right|\;\;\;\mbox{for}\;\;\;x,y\in X.\end{equation*}
Finally, since $l_0\in\mathbb{N}$ was chosen arbitrarily, and, by assumption of this lemma, \linebreak\hbox{$\sum_{l=1}^{\infty}\mathbb{E}_{x,y}\left|z_l^{(1)}-z_l^{(2)}\right|<\infty$} for any $x,y\in X$, we can conclude that
\begin{equation*}\left|f_{m,c}^{\inf}(x)-f_{m,c}^{\inf}(y)\right|=0\;\;\;\mbox{for any}\;\;\;x,y\in X.\end{equation*} The proof is therefore completed.
\end{proof}

\begin{proof}[Proof of Lemma \ref{lem:continuous}]
Fix an arbitrary $m\in\mathbb{N}\cup\{\infty\}$ and let 
$c_m:=\mathbb{E}_{\mu_*}\left(z_1^2\wedge m\right)$.  
According to Lemma \ref{lem:birkhoff}, we know that 
\begin{align*}
\lim_{n\to\infty} \frac{1}{n}\sum_{l=1}^n \left(z_l^2\wedge m\right)=\lim_{n\to\infty} \frac{1}{n}\sum_{l=1}^n \left(z_1^2\wedge m\right)\circ T^{l-1}
=\mathbb{E}_{\mu_*}\left(z_1^2\wedge m\right)=c_m\;\;\;\mathbb{P}_{\mu_*}\mbox{-a.s.},
\end{align*}
which ensures
\begin{align}
\int_X f^{\inf}_{m,c_m}(x)\,\mu_*(dx)=\mathbb{E}_{\mu_*}\left(\left|\liminf_{n\to\infty}\left(\frac{1}{n}\sum_{l=1}^n \left(z_l^2\wedge m\right) \right)-c_m\right|\wedge 1\right)=0,\label{m_star1}
\\
\int_X f^{\sup}_{m,c_m}(x)\,\mu_*(dx)=\mathbb{E}_{\mu_*}\left(\left|\limsup_{n\to\infty}\left(\frac{1}{n}\sum_{l=1}^n \left(z_l^2\wedge m\right) \right)-c_m\right|\wedge 1\right)=0.\label{m_star2}
\end{align}
Referring to \eqref{eq:m_shift} and using the fact that $z_l^2 \circ T^N = z_{l+N}^2$, for any $N\in\mathbb{N}$, we have
\begin{align*}
&\mathbb{E}_{\mu}\left(\left|\liminf_{n\to\infty}\left(\frac{1}{n}\sum_{l=1}^n \left(z_l^2\wedge m\right) \right)-c_m\right|\wedge 1\right)\\
&\qquad=\mathbb{E}_{\mu}\left(\mathbb{E}_{\mu}\left(\left(\left|\liminf_{n\to\infty}\left(\frac{1}{n}\sum_{l=1}^n \left(z_l^2\wedge m\right) \right)-c_m\right|\wedge 1\right)\circ T^N\,\Big|\, \mathcal{F}_N\right)\right)\\
&\qquad=\mathbb{E}_{\mu}\left(\mathbb{E}_{\phi_N}\left(\left|\liminf_{n\to\infty}\left(\frac{1}{n}\sum_{l=1}^n \left(z_l^2\wedge m\right) \right)-c_m\right|\wedge 1\right)\right)
=\int_{X} f_{m,c_m}^{\inf}(x) \, P^N\mu(dx).
\end{align*}
Similar reasoning leads to
\begin{equation*}\mathbb{E}_{\mu}\left(\left|\limsup_{n\to\infty}\left(\frac{1}{n}\sum_{l=1}^n \left(z_l^2\wedge m\right) \right)-c_m\right|\wedge 1\right)=\int_{X} f_{m,c_m}^{\sup}(x)\,P^N\mu(dx)
\;\;\;\text{for any}\;\;\;N\in\mathbb{N}.
\end{equation*}

The functions $f_{m,c_m}^{\inf}$ and $f_{m,c_m}^{\sup}$ are obviously bounded, and, by the  assumption of the lemma, they are also continuous. Therefore, using the fact that $(P^n\mu)_{n\in\mathbb{N}}$ converges weakly to $\mu_*\in\mathcal{M}_1(X)$, and applying identities \eqref{m_star1}, \eqref{m_star2}, we can deduce that
\begin{align*}
\mathbb{E}_{\mu}\left(\left|\liminf_{n\to\infty}\left(\frac{1}{n}\sum_{l=1}^n \left(z_l^2\wedge m\right) \right)-c_m\right|\wedge 1\right)
&=\lim_{N\to\infty}\int_{X} f_{m,c_m}^{\inf}(x)\,P^N\mu(dx)\\
&=\int_{X} f_{m,c_m}^{\inf}(x)\,\mu_*(dx)=0,
\end{align*}
and analogously
\begin{align*}
\mathbb{E}_{\mu}\left(\left|\limsup_{n\to\infty}\left(\frac{1}{n}\sum_{l=1}^n \left(z_l^2\wedge m\right) \right)-c_m\right|\wedge 1\right)
&=\lim_{N\to\infty}\int_{X} f_{m,c_m}^{\sup}(x)\,P^N\mu(dx)\\
&=\int_{X} f_{m,c_m}^{\sup}(x)\,\mu_*(dx)=0.
\end{align*}
Finally, we obatin
\begin{align*}
\liminf_{n\to\infty} \frac{1}{n}\sum_{l=1}^n \left(z_l^2 \wedge m\right) 
= \limsup_{n\to\infty} \frac{1}{n}\sum_{l=1}^n \left(z_l^2 \wedge m\right)=c_m   \;\;\;\mathbb{P}_{\mu}\mbox{-a.s},
\end{align*}
which ends the proof.
\end{proof}

\begin{proof}[Proof of Lemma \ref{lem:dawida}]
According to Lemmas \ref{lem:constant} and \ref{lem:continuous} we know that, for any $m\in\mathbb{N}$,
\begin{align*}
\lim_{n\to\infty}\frac{1}{n}\sum_{l=1}^n\left(z_l^2\wedge m\right)=\mathbb{E}_{\mu_*}\left(z_1^2\wedge m\right)\;\;\;\mathbb{P}_{\mu}\mbox{-a.s.}
\end{align*}
Moreover, we see that $\left(n^{-1}\sum_{l=1}^n(z_l^2\wedge m)\right)_{n\in\mathbb{N}}$ is bounded by $m$, and therefore, due to the Dominated Convergence Theorem, we can conclude that,  for $m\in\mathbb{N}$, 
\begin{align}\label{eq:wedge_m}
\lim_{n\to\infty}\frac{1}{n}\sum_{l=1}^n\mathbb{E}_{\mu}\left(z_l^2\wedge m\right)=\mathbb{E}_{\mu}\left(\lim_{n\to\infty}\frac{1}{n}\sum_{l=1}^n\left(z_l^2\wedge m\right)\right)=\mathbb{E}_{\mu_*}\left(z_1^2\wedge m\right).
\end{align}
Further, note that, for $m\in\mathbb{N}$, 
\begin{align*}
\sup_{l\in\mathbb{N}}\mathbb{E}_{\mu}\left(z_l^2\mathbbm{1}_{\{z_l^2\geq m\}}\right)
= \sup_{l\in\mathbb{N}}\mathbb{E}_{\mu}\left(|z_l|^{2+r}\left(z_l^2\right)^{-r/2}\mathbbm{1}_{\{z_l^2\geq m\}}\right)
\leq m^{-r/2}\sup_{l\in\mathbb{N}}\mathbb{E}_{\mu}\left|z_l\right|^{2+r}.
\end{align*}
Hence, according to assumption \eqref{cond:b}, we infer that
\begin{align}\label{eq:to0}
\sup_{l\in\mathbb{N}}\mathbb{E}_{\mu}\left(z_l^2\mathbbm{1}_{\{z_l^2\geq m\}}\right)\to 0,\;\;\;\mbox{as}\;\;\;m\to\infty.
\end{align}
Now, observe that, for every $l\in\mathbb{N}$ and any $m\in\mathbb{N}$, 
\begin{align*}
&\big|\mathbb{E}_{\mu}\left(z_l^2\wedge m\right)-\mathbb{E}_{\mu}\left(z_l^2\right)\big|\\
&\qquad
\leq\left|\mathbb{E}_{\mu}\left(z_l^2\wedge m\right)-\mathbb{E}_{\mu}\left(z_l^2\mathbbm{1}_{\{z_l^2< m\}}\right)\right|+\mathbb{E}_{\mu}\left(z_l^2\mathbbm{1}_{\{z_l^2\geq m\}}\right)\\
&\qquad =\left|\mathbb{E}_{\mu}\left(m\mathbbm{1}_{\{z_l^2\geq m\}}\right)+\mathbb{E}_{\mu}\left(z_l^2\mathbbm{1}_{\{z_l^2< m\}}\right)-\mathbb{E}_{\mu}\left(z_l^2\mathbbm{1}_{\{z_l^2< m\}}\right)\right|
+\mathbb{E}_{\mu}\left(z_l^2\mathbbm{1}_{\{z_l^2\geq m\}}\right)\\
&\qquad=\mathbb{E}_{\mu}\left(m\mathbbm{1}_{\{z_l^2\geq m\}}\right)+\mathbb{E}_{\mu}\left(z_l^2\mathbbm{1}_{\{z_l^2\geq m\}}\right)
\leq 2\mathbb{E}_{\mu}\left(z_l^2\mathbbm{1}_{\{z_l^2\geq m\}}\right)
\leq 2\sup_{l\in\mathbb{N}}\mathbb{E}_{\mu}\left(z_l^2\mathbbm{1}_{\{z_l^2\geq m\}}\right),
\end{align*}
which gives, for any $n \in \mathbb{N}$,
\begin{align*}
\left|\frac{1}{n}\sum_{l=1}^n\mathbb{E}_{\mu}\left(z_l^2\wedge m\right)-\frac{1}{n}\sum_{l=1}^n\mathbb{E}_{\mu}\left(z_l^2\right)\right|
&\leq \frac{1}{n}\sum_{l=1}^n\left|\mathbb{E}_{\mu}\left(z_l^2\wedge m\right)-\mathbb{E}_{\mu}\left(z_l^2\right)\right|
\leq 2\sup_{l\in\mathbb{N}}\mathbb{E}_{\mu}\left(z_l^2\mathbbm{1}_{\{z_l^2\geq m\}}\right).
\end{align*}
Consequently, for any $m\in\mathbb{N}$, we can write
\begin{align*}
\left|\mathbb{E}_{\mu_*}\left(z_1^2\right)-\frac{1}{n}\sum_{l=1}^n\mathbb{E}_{\mu}\left(z_l^2\right)\right|
&\leq \left|\mathbb{E}_{\mu_*}\left(z_1^2\right)-\mathbb{E}_{\mu_*}\left(z_1^2\wedge m\right)\right|
+\left|\mathbb{E}_{\mu_*}\left(z_1^2\wedge m\right)-\frac{1}{n}\sum_{l=1}^n\mathbb{E}_{\mu}\left(z_l^2\wedge m\right)\right|\\
&\quad+2\sup_{l\in\mathbb{N}}\mathbb{E}_{\mu}\left(z_l^2\mathbbm{1}_{\{z_l^2\geq m\}}\right),
\end{align*}
and hence, due to \eqref{eq:wedge_m}, we have
\begin{align*}
\limsup_{n\to\infty}\left|\frac{1}{n}\sum_{l=1}^n\mathbb{E}_{\mu}\left(z_l^2\right)-\mathbb{E}_{\mu_*}\left(z_1^2\right)\right|
&\leq 2\sup_{l\in\mathbb{N}}\mathbb{E}_{\mu}\left(z_l^2\mathbbm{1}_{\{z_l^2\geq m\}}\right)
+\left|\mathbb{E}_{\mu_*}\left(z_1^2\wedge m\right)-\mathbb{E}_{\mu_*}\left(z_1^2\right)\right|
\end{align*}
for all $m\in\mathbb{N}$. 
Now, we see that the right-hand side of the above inequality tends to zero, as $m\to\infty$, which follows from \eqref{eq:to0}. 
Finally, by the orthogonality of martingale differences, we get
\begin{align*}
\lim_{n\to\infty}\frac{h_n^2(\mu)}{n}=\lim_{n\to\infty}\frac{1}{n}\mathbb{E}_{\mu}\left(m_n^2\right)=\lim_{n\to\infty}\frac{1}{n}\sum_{l=1}^n\mathbb{E}_{\mu}\left(z_l^2\right)=\mathbb{E}_{\mu_*}\left(z_1^2\right)=\sigma^2,
\end{align*}
which completes the proof of \eqref{eq:h_n_sigma}. 

Now, in order to establish \eqref{eq:h_n_1}, it is enough to 
observe that \eqref{eq:h_n_sigma} and Lemma \ref{lem:continuous} imply 
\begin{align*}
\lim_{n\to\infty}\frac{1}{h_n^2(\mu)}\sum_{l=1}^nz_l^2
=\frac{1}{\sigma^2}\lim_{n\to\infty}\frac{1}{n}\sum_{l=1}^nz_l^2=1.
\end{align*}
\end{proof}

\begin{proof}[Proof of Lemma \ref{lem:ogolny}]
Let $\upsilon, \vartheta>0$. From Lemma \ref{lem:dawida}, we can deduce that $h_n(\mu)>0$ for $n\geq N$, where $N$ is some sufficiently large constant. Now, since $r\in(0,2)$, we obtain, for $n\geq N$,  
\begin{align*}
h_n^{-4}(\mu)\mathbb{E}_{\mu}\left(z_n^4\mathbbm{1}_{\{|z_n|<\,\upsilon\, h_n(\mu)\}}\right)
&\leq h_n^{-4}(\mu)\mathbb{E}_{\mu}\left(|z_n|^{2+r}\upsilon^{2-r}h_n^{2-r}(\mu)\right)
\leq \upsilon^{2-r}\left(\sup_{n\in\mathbb{N}}\mathbb{E}_{\mu}|z_n|^{2+r}\right)h_n^{-2-r}(\mu),
\end{align*}
and similarly
\begin{align*}
h_n^{-1}(\mu)\mathbb{E}_{\mu}
\left(|z_n|\mathbbm{1}_{\{|z_n|\geq \,\vartheta \,h_n(\mu)\}}\right)
&\leq h_n^{-1}(\mu)\mathbb{E}_{\mu}\left(\frac{|z_n|^{2+r}}{\left(\vartheta\, h_n(\mu)\right)^{1+r}}\right)
\leq\vartheta^{-1-r}\left(\sup_{n\in\mathbb{N}}\mathbb{E}_{\mu}|z_n|^{2+r}\right)h_n^{-2-r}(\mu).
\end{align*}
Since $\sup_{n\in\mathbb{N}}\mathbb{E}_{\mu}|z_n|^{2+r}<\infty$, and, due to Lemma \ref{lem:dawida}, $\sum_{n=N}^{\infty}h_n^{-2-r}(\mu)<\infty$, the proof is completed.
\end{proof}

\section*{Acknowledgements}
Hanna Wojew\'odka-\'Sci\k{a}\.zko was supported by the Foundation for Polish Science (FNP). Part of this work was done when Hanna Wojew\'odka-\'Sci\k{a}\.zko attended a four-week study trip to the Mathematical Institute at Leiden University, which was also supported by the FNP (the so-called `Outgoing Stipend' in the START programme).


\bibliographystyle{plain}
\bibliography{references}        

%

\end{document}